\documentclass[a4paper,english,11pt]{article}

\usepackage[a4paper,left=2.5cm,right=2.5cm,top=2.5cm,bottom=2.5cm]{geometry}
\usepackage{bbm}
\usepackage{latexsym,amsthm,amsmath,amscd,amssymb,color,url,amsfonts}
\usepackage[utf8]{inputenc}
\usepackage[T1]{fontenc} 
\usepackage{mathrsfs}
\usepackage[english]{babel}
\usepackage{libertine}
\usepackage{caption} 
\usepackage{graphicx}
\usepackage{MnSymbol}
\usepackage{hyperref}
\usepackage{setspace}
\usepackage{chngpage}

\newtheorem{lemma}{Lemma}[section]
\newtheorem{thm}[lemma]{Theorem}
\newtheorem{thmintro}{Theorem}[section]
\newtheorem{corointro}[thmintro]{Corollary}
\newtheorem*{thm*}{Theorem}

\newtheorem{rem}[lemma]{Remark}
\newtheorem{propo}[lemma]{Proposition}
\newtheorem{propr}[lemma]{Property}

\newtheorem{cor}[lemma]{Corollary}

\newtheorem{example}[lemma]{Example}

\newtheorem{defn}[lemma]{Definition}

\newtheorem{nota}[lemma]{Notation}

\title{Geodesics and visual boundary of horospherical products}

\author{Tom Ferragut}

\date{\today}

\begin{document}
\maketitle

\begin{abstract}
We study the geometry of \textit{horospherical products} by providing a description of their distances, geodesics and visual boundary. These products contains both discrete and continuous examples, including Cayley graphs of lamplighter groups and solvable Lie groups of the form $\mathbb{R}\ltimes (N_1\times N_2)$, where $N_1$ and $N_2$ are two simply connected, nilpotent Lie groups.
\end{abstract}

\tableofcontents

\section{Introduction}

A horospherical product is a metric space constructed from two Gromov hyperbolic spaces $X$ and $Y$, it is included in their Cartesian product $X\times Y$ and can be seen as a diagonal in it. Let $\beta_X:X\to\mathbb{R} $ and $\beta_Y:Y\to \mathbb{R}$ be two Busemann functions. The horospherical product of $X$ and $Y$, denoted by $X\bowtie Y$, is defined as the set of points in $X\times Y$ such that the two Busemann functions add up to zero, namely
\begin{equation}
X\bowtie Y:=\lbrace (x,y)\in X\times Y\ /\ \beta_X(x)+\beta_Y(y)=0 \rbrace.\nonumber
\end{equation}
The level-lines of the Busemann functions are called \textit{horospheres}, one can see the horospherical product $X\bowtie Y$ as $X$ crossed with an upside down copy of $Y$ in parallel to these horospheres. We will call height function the opposite of the chosen Busemann function.
\\\\Let $N$ be a simply connected, nilpotent Lie group and let $A$ be a derivation of $\mathrm{Lie}(N)$ whose eigenvalues have positive real parts. Then $\mathbb{R}\ltimes_A N$ is called a Heintze group and is Gromov hyperbolic, they are the only examples of negatively curved Lie groups. Let $X$ and $Y$ are two Heintze groups, we can choose the Busemann functions to be such that $\forall (t,n)\in \mathbb{R}\ltimes_A N$ we have $\beta(t,n)=-t$. Then we obtain $$(\mathbb{R}\ltimes_{A_1} N_1)\bowtie (\mathbb{R}\ltimes_{A_2} N_2)=\mathbb{R}\ltimes_{\mathrm{Diag}(A_1,-A_2)} (N_1\times N_2).$$
When $N=\mathbb{R}$, the corresponding Heintze group is a hyperbolic plan $\mathbb{H}^2$, and as their horospherical products we obtain the Sol geometries, one of the eight Thurston's geometries. We can also build Diestel-Leader graphs and the Cayley 2-complexes of Baumslag-Solitar groups $\mathrm{BS}(1,n)$ as the horospherical products of trees or hyperbolic plans. In the second section of \cite{Woess}, the last three sets of examples are well detailed, and presented as horocyclic products of either regular trees or the hyperbolic plan $\mathbb{H}^2$. We choose the name horospherical product instead of horo\textit{cyclic} product since in higher dimension, level-sets according to a Busemann function are not horocycles but horospheres. 

As Woess suggested in the end of \cite{Woess}, we explore here a generalization for horospherical products. The horospherical product construction can be realized for more than two spaces, see \cite{BNW} for a study of the Brownian motion on a multiple horospherical product of trees. However in this work we will stay in the setting of two Gromov hyperbolic spaces.
\\\\To study the geometry of horospherical products we require that our components $X$ and $Y$ are two proper, geodesically complete, Gromov hyperbolic, Busemann spaces. A Busemann space is a metric space where the distance between any two geodesics is convex, and a metric space $X$ is geodesically complete if and only if a geodesic segment $\alpha:I\to X$ can be prolonged into a geodesic line $\hat{\alpha} :\mathbb{R}\to X$. The Busemann hypothesis suits with the definition of horospherical product since we require the two heights functions to be exactly opposite. Furthermore, adding the assumptions that $X$ and $Y$ are geodesically complete allows us to prove that the horospherical product $X\bowtie Y$ is connected (see Lemma \ref{HoroProdConnected}). 
\\\\In the next part of this introduction we present our main results, which hold when $X$ and $Y$ are two proper, geodesically complete, Gromov hyperbolic, Busemann spaces. It covers the case where $X$ and $Y$ are solvable Lie groups of the form $\mathbb{R}\ltimes_A N$.

In \cite{FB1} and \cite{FB2}, using the horospherical product structure of treebolic space, Farb and Mosher proved a rigidity results for quasi-isometries of $\mathrm{BS}(1,n)$. In \cite{EFW1} and \cite{EFW2}, Eskin, Fisher and Whyte obtained a similar rigidity results for the Diestel-Leader graphs and the Sol geometries, again using their horospherical product structure.

Besides being results on their own, the tools we develop in this paper are used in \cite{TF2} to study the quasi-isometry classification of the aforementioned horospherical products. In \cite{TF2} we generalise the results obtained by Eskin, Fisher and Whyte in \cite{EFW1}, and provide new quasi-isometric classifications for some family of solvable Lie groups.
\\\\There are many possible choices for the distance on $X\bowtie Y$ in this paper we work with a family of length path metrics induced by distances on $X\times Y$ (see Definition \ref{DefNDistHoro}). We require that the distance on $X\bowtie Y$ comes from an \textit{admissible} norm $N$ on $\mathbb{R}^2$ (e.g. any $\ell_p$ norm). Our first result describes these distances.
\begin{thmintro}\label{THMDist}
Let $d_{\bowtie}$ be an admissible distance on $X\bowtie Y$. Then there exists a constant $M$ depending only on the metric spaces $(X\bowtie Y,d_{\bowtie})$ such that for all $p=(p^X,p^Y),q=(q^X,q^Y)\in X\bowtie Y$:
\begin{align*}
\Big|d_{\bowtie}(p,q)-\Big(d^X(p^X,q^X)+d^Y(p^Y,q^Y)-|h(p)-h(q)|\Big)\Big|\leq M.
\end{align*}
\end{thmintro}
Therefore, given two admissible distances $d$ and $d'$, the horospherical products $(X\bowtie Y,d)$ and $(X\bowtie Y,d')$ are roughly isometric, which means that there exists a $(1,c)$-quasi-isometry between them, for a constant $c\geq 0$. Le Donne, Pallier and Xie proved in \cite{LDPX} that for the solvable groups $\mathbb{R}\ltimes_{\mathrm{Diag}(A_1,-A_2)} (N_1\times N_2)$, changing the left-invariant Riemannian metric results in the identity map being a rough similarity.

Theorem \ref{THMDist} is one of the tools we use in \cite{TF2}, where we prove a geometric rigidity property of quasi-isometries between families of horospherical products. This property leads to quasi-isometric invariants in such spaces, and a first result in the quasi-isometry classification of some solvable Lie groups.
\\\\Throughout this paper we provide a coarse description of geodesics and of the visual boundary of a broad family of horospherical products. 

Following the characterisation of the distances on horospherical products, we describe the shape of geodesic segments. 

\begin{thmintro}\label{THMAINTRO}
Let $X$ and $Y$ be two proper, geodesically complete, $\delta$-hyperbolic, Busemann spaces and let $d_{\bowtie}$ be an admissible distance on $X\bowtie Y$. Let $p=\left(p^X,p^Y\right)$ and $q=\left(q^X,q^Y\right)$ be two points of $X\bowtie Y$ and let $\alpha$ be a geodesic segment of $(X\bowtie Y,d_{\bowtie})$ linking $p$ to $q$. There exists a constant $M$ depending only on $(X\bowtie Y,d_{\bowtie})$, and there exist two vertical geodesics $V_1=\left(V^X_1,V^Y_1\right)$ and $V_2=\left(V^X_2,V^Y_2\right)$ such that: 
\begin{enumerate}
\item If ~ $h(p)\leq h(q)-M$ ~ then $\alpha$ is in the $M$-neighbourhood of $V_1\cup\left(V^X_1,V^Y_2\right)\cup V_2$;
\item If ~ $h(p)\geq h(q)+M$ ~ then $\alpha$ is in the $M$-neighbourhood of $ V_1\cup\left(V^X_2,V^Y_1\right)\cup V_2$;
\item If ~ $|h(p)-h(q)|\leq M$ ~ then at least one of the conclusions of $1.$ or $2.$ holds.
\end{enumerate}
Specifically $V_1$ and $V_2$ can be chosen such that $p$ is close to $V_1$ and $q$ is close to $V_2$.
\end{thmintro}

\begin{figure}
\begin{center}
\includegraphics[scale=0.60]{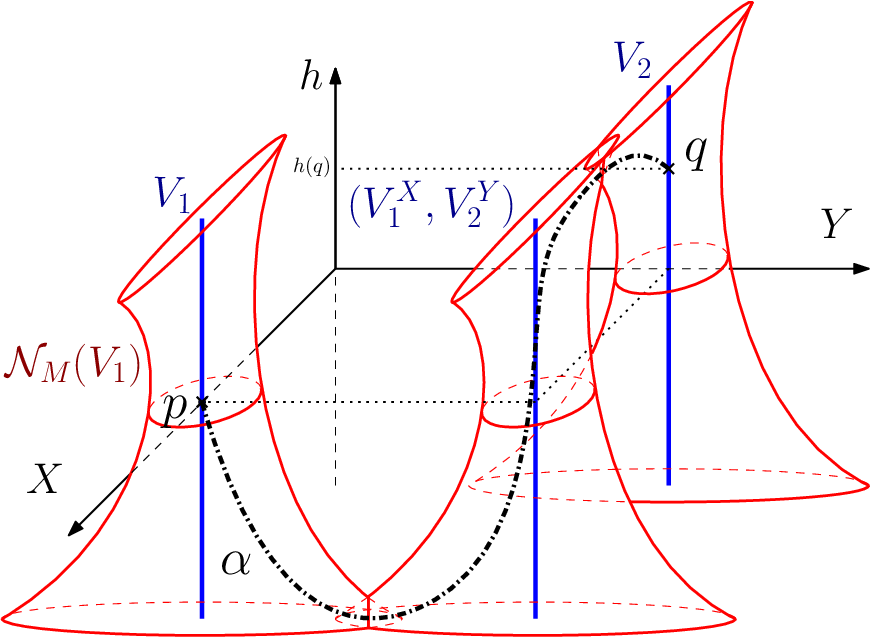}
\captionof{figure}{Shape of geodesic segments when $h(p)\leq h(q)-\kappa$ in $X\bowtie Y$. The neighbourhoods' shapes  are distorted since when going upward, distances are contracted in the "direction" $X$ and expanded in the "direction" $Y$.}
\label{FigTHMAINTRO}
\end{center}
\end{figure}

An example is illustrated on Figure \ref{FigTHMAINTRO} for $h(p)\leq h(q)-\kappa$. Coarsely speaking, Theorem \ref{THMAINTRO} ensures that any geodesic segment is constructed as the concatenation of three vertical geodesics. This result is similar to the Gromov hyperbolic case, where a geodesic segment is in the constant neighbourhood of two vertical geodesics. This result leads us to the existence of unextendable geodesics, which are called \textit{dead-ends}. Geodesics shapes was already well-known in lamplighter groups. In the case of Sol, we recover, up to an additive constant, Troyanov's description of global geodesics (see \cite{Troya}).
\\The horospherical product between $X$ and $\mathbb{R}$ is isometric to $X$, therefore given any vertical geodesic $V^Y$ of $Y$, $X\bowtie V^Y$ is an embedded copy of $X$ in $X\bowtie Y$. A geodesic line of $X\bowtie Y$ looks either like a geodesic of $X$ or like a geodesic of $Y$.

\begin{figure}
\begin{center}
\includegraphics[scale=0.85]{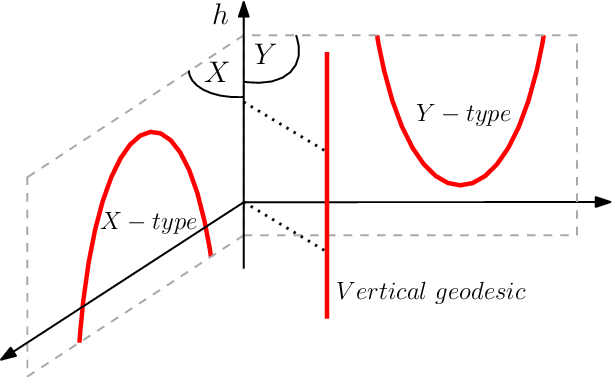}
\captionof{figure}{Different types of geodesics in $X\bowtie Y$.}
\label{FinalThmInIntro}
\end{center}
\end{figure}

\begin{corointro}\label{THMBINTRO}
Let $X$ and $Y$ be two proper, geodesically complete, $\delta$-hyperbolic, Busemann spaces. Then there exists $M\geq 0$ depending only on $\delta$ such that for all geodesic line $\alpha:\mathbb{R}\rightarrow X\bowtie Y$ at least one of the two following statements holds.
\begin{enumerate}
\item $\alpha$ is included in a constant $M$-neighbourhood of a geodesics contained in a embedded copy of $X$;
\item $\alpha$ is included in a constant $M$-neighbourhood of a geodesics contained in a embedded copy of $Y$.
\end{enumerate}
\end{corointro}

If a geodesic verifies both conclusions, it is in the $M$-neighbourhood of a vertical geodesic of $X\bowtie Y$. 
\\\\Let $o\in X\bowtie Y$, the visual boundary of $X\bowtie Y$ with respect to the base point $o$, denoted by $\partial_o(X\bowtie Y)$, stands for the set of equivalence classes of geodesic rays starting at $o$. Consequently to the description of geodesic segments, we obtain that for any geodesic ray $k$ of $X\bowtie Y$ there exists a vertical geodesic ray at finite distance of $k$. Therefore we classify all possible shapes for geodesic rays, then we give a description of the visual boundary of $X\bowtie Y$. 

\begin{thmintro}\label{THMCINTRO}
Let $X$ and $Y$ be two proper, geodesically complete, $\delta$-hyperbolic, Busemann spaces. Let $(w^X,a^X)\in X\times\partial X$, $(w^Y,a^Y)\in Y\times\partial Y$ and let $X\bowtie Y$ be the horospherical product with respect to $(w^X,a^X)$ and $(w^Y,a^Y)$. Then the visual boundary of $X\bowtie Y$ with respect to any point $o=(o^X,o^Y)$ can be decomposed as:
\begin{align*}
\partial_o (X\bowtie Y)=&\Big(\big(\partial X\setminus  \lbrace a^X\rbrace\big)\times\lbrace a^Y\rbrace\Big)\bigcup\Big(\lbrace a^X\rbrace\times\big(\partial Y\setminus \lbrace a^Y\rbrace\big) \Big)
\\=&\Big(\big(\partial X\times\lbrace a^Y\rbrace\big)\bigcup\big(\lbrace a^X\rbrace\times\partial Y\big)\Big)\setminus \lbrace(a^X,a^Y)\rbrace 
\end{align*}
\end{thmintro}

\begin{figure}
\begin{center}
\includegraphics[scale=0.85]{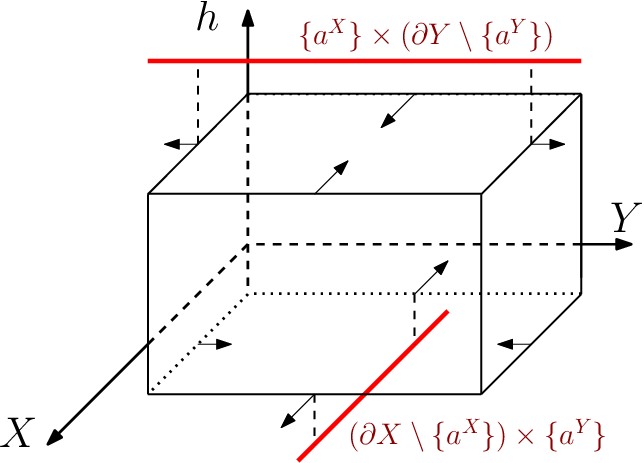}
\captionof{figure}{Depiction of $\partial_o (X\bowtie Y)$.}
\label{FinalThmBoundary}
\end{center}
\end{figure}
When $X:=\mathbb{R}\ltimes_{A_1} N_1$ and $Y:=\mathbb{R}\ltimes_{A_2} N_2$ we obtain that
\begin{align*}
\partial\left(\mathbb{R}\ltimes_{\mathrm{Diag}(A_1,-A_2)} (N_1\times N_2)\right)=N_1\times N_2.
\end{align*}
In the case of Sol, the last result is similar to Proposition 6.4 of \cite{Troya}. However, unlike Troyanov in his work, we are focusing on minimal geodesics and not on local ones. One can see that this visual boundary neither depends on the chosen admissible distance $d$ nor on the base point $o$.

\subsection*{Framework}

The paper is organized as follows.

\begin{enumerate}
\item[$\bullet$] In Section \ref{SectContext} we present the context in which we will construct the horospherical products, namely Gromov hyperbolic, Busemann spaces. 
\item[$\bullet$] Then in Section \ref{SectHoroDefn} we define horospherical products and give some examples.
\item[$\bullet$] In Section \ref{SecMetricEstimat} we present an estimate on the length of paths avoiding horoballs in hyperbolic spaces, namely Lemma \ref{ExpLengthWhenBelowAndReachPoint}, which will be central in our control of the distances on $X\bowtie Y$.
Then we give an estimate of the distances on $X\bowtie Y$ through Theorem \ref{lengthGeod}.
\item[$\bullet$] Last, in Section \ref{SectionThm} , we prove our main results, Theorem \ref{THMDist} follows from Corollary \ref{lengthGeod}. The description of geodesic lines of Theorem \ref{THMAINTRO} follows from Theorem \ref{THMDist} and gives us the tools to prove Theorem \ref{THMCINTRO}.
\end{enumerate}

\section*{Acknowledgement}

This work was supported by the University of Montpellier. 

I thank my two advisers Jeremie Brieussel and Constantin Vernicos for their many relevant reviews and comments.

\newpage
\section{Context}\label{SectContext}

The goal of this section is to present what is a Gromov hyperbolic, Busemann space and what are vertical geodesics in such a space. Let $(H,d_H)$ be a proper, geodesic, metric space. 

\subsection{Gromov hyperbolic spaces}

A geodesic line, respectively ray, segment, of $H$ is the isometric image of a Euclidean line, respectively half Euclidean line, interval, in $H$. By slight abuse, we may call geodesic, geodesic ray or geodesic segment, the map $\alpha:I\rightarrow H$ itself, which parametrises our given geodesic by arclength.
\\\\Let $\delta\geq 0$ be a non-negative number. Let $x$, $y$ and $z$ be three points of $H$. The geodesic triangle $[x,y]\cup[y,z]\cup[z,x]$ is called $\delta$-slim if any of its sides is included in the $\delta$-neighbourhood of the remaining two. The metric space $H$ is called $\delta$-hyperbolic if every geodesic triangle is $\delta$-slim. A metric space $H$ is called Gromov hyperbolic if there exists $\delta\geq 0$ such that $H$ is a $\delta$-hyperbolic space.
\\\\An important property of Gromov hyperbolic spaces is that they admit a nice compactification thanks to their \textit{Gromov boundary}. We call two geodesic rays of $H$ equivalent if their images are at finite Hausdorff distance. Let $w\in H$ be a base point. We define $\partial_{w} H$ the Gromov boundary of $H$ as the set of families of equivalent rays starting from $w$. The boundary $\partial_{w} H$ does not depend on the base point $w$, hence we will simply denote it by $\partial H$. Both $\partial H$ and $H\cup \partial H$, are compact endowed with the Hausdorff topology. For more details, see \cite{GDLH} or chap.III H. p.399 of \cite{BH}.
\\\\Let us fix a point $a\in\partial H$ on the boundary. We call \textbf{vertical geodesic ray}, respectively \textbf{vertical geodesic line}, any geodesic ray in the equivalence class $a$, respectively with one of its rays in $a$. The study of these specific geodesic rays is central in this work. 

\subsection{Busemann spaces and Busemann functions}

A metric space $(H,d_H)$ is \textbf{Busemann} if and only if for every pair of geodesic segments parametrized by arclength $\gamma:[a,b]\rightarrow H$ and $\gamma ':[a',b']\rightarrow H$, the following function is convex:
\begin{align*}
D_{\gamma ,\gamma '}:[a,b]\times[a',b']&\rightarrow H 
\\(t,t')&\mapsto d_H(\gamma(t),\gamma '(t')).
\end{align*}
It is a weaker assumption than being CAT$(0)$ (Theorem $1.3$ of \cite{FLS}), however it implies that $H$ is uniquely geodesic. See Chap.8 and Chap.12 of \cite{Papa} for more details on Busemann spaces. 
\\This convex assumption removes some technical difficulties in a significant number of proofs in this work. If $H$ is a Busemann space in addition to being Gromov hyperbolic, for all $x\in H$ there exists a unique vertical geodesic ray, denoted by $V_x$, starting at $H$. In  fact the distance between two vertical geodesics starting at $x$ is a  convex and bounded function, hence decreasing and therefore constant equal to $0$.
\\The construction of the \textit{horospherical product} of two Gromov hyperbolic space $X$ and $Y$ requires the so called \textbf{Busemann functions}. Their definition is simplified by the Busemann assumption. Let us consider $\partial X$, the Gromov boundary of $X$ (which, in this setting, is the same as the visual boundary). Both the boundary $\partial X$ and $X\cup \partial X$, endowed with the natural Hausdorff topology, are compact. Then, given $a\in\partial X$ a point on the boundary, and $w\in X$ a base point, we define a Busemann function $\beta_{(a,w)}$ with respect to $a$ and $w$. Let $V_w$ be the unique vertical geodesic ray starting from $w$.
\begin{equation}
\forall\ x \in X,\ \beta_{(a,w)}(x):= \limsup\limits_{t\rightarrow + \infty}(d(x,V_w(t))-t)\quad.\nonumber
\end{equation}
This function computes the asymptotic delay a point $x\in X$ has in a race towards $a$ against the vertical geodesic ray starting at $w$. 
The \textbf{horospheres} of $X$ with respect to $(a,w)\in \partial X\times X$ are the level-sets of $\beta_{(a,w)}$. These horospheres depend on the previously chosen couple $(a,w)$ of $\partial X\times X$.

\subsection{Heights functions and vertical geodesics}\label{SecHeight}

In this section we fix $\delta\geq 0$, $H$ a proper, geodesic, $\delta$-hyperbolic space, $w\in H$ a base point and $a\in\partial H$ a point on the boundary of $H$. We call \textbf{height function}, denoted by $h$, the opposite of the Busemann function, $h:=-\beta_{(a,w)}$.

Let us write Proposition 2 chap.8 p.136 of \cite{GDLH} with our notations.

\begin{propo}[\cite{GDLH}, chap.8 p.136]\label{PropHautGDLH}
Let $H$ be a hyperbolic proper geodesic metric space. Let $a\in \partial H$ and $w\in H$, then:
\begin{enumerate}
\item $ \lim\limits_{x\rightarrow a}h_{(a,w)}(x)=+\infty$
\item $ \lim\limits_{x\rightarrow b}h_{(a,w)}(x)=-\infty$, $\forall b \in \partial H \setminus \lbrace a\rbrace$
\item $\forall x,y,z\in H, | \beta_a(x,y)+\beta_a(y,z)-\beta_a(x,z)|\leq 200\delta$.
\end{enumerate}
\end{propo}

Furthermore, a geodesic ray is in $a\in\partial H$ if and only if its height tends to $+\infty$.

\begin{cor}\label{PropoGeodVertVersa}
Let $H$ be a hyperbolic proper geodesic metric space. Let $a\in \partial H$ and $w\in H$, and let $\alpha:[0,+\infty[\to H$ be a geodesic ray. The two following properties are equivalent:
\begin{enumerate}
\item $\lim\limits_{t\rightarrow +\infty} h_{(a,w)}(\alpha(t))= +\infty$
\item $\alpha([0,+\infty[)\in a$.
\end{enumerate}
\end{cor}

\begin{proof}\ 
As for any geodesic ray $\alpha:[0,+\infty[\to H$ there exists $b\in\partial H$ such that $\alpha ([0,+\infty[)\in b$, this proposition is a particular case of Proposition \ref{PropHautGDLH}.
\end{proof}

An important property of the height function is to be Lipschitz.

\begin{propo}\label{PropHautLipsch}
Let $a\in\partial H$ and $w\in H$. The height function $h_a:=-\beta_a (\cdot,w)$ is Lipschitz:
\begin{equation}
\forall x,y\in H, |h_{(a,w)}(x)-h_{(a,w)}(y)|\leq d (x,y).\nonumber
\end{equation}
\end{propo}
\begin{proof}
By using the triangle inequality we have for all $x,y\in H$:
\begin{align}
-h_{(a,w)}(x)&=\beta_a (x,w)=\sup\lbrace \limsup\limits_{t\rightarrow + \infty}(d(x,k(t))-t)\mid \text{k vertical rays starting at } w \rbrace\nonumber
\\&\leq d(x,y)+\sup\lbrace \limsup\limits_{t\rightarrow + \infty}(d( y,k(t))-t)\mid \text{k vertical rays starting at } w \rbrace \nonumber
\\&\leq d(x,y)+\beta_a (y,w) \leq d(x,y)-h_{(a,w)} (y).\nonumber
\end{align}
The result follows by exchanging the roles of $x$ and $y$.
\end{proof}

From now on, we fix a given $a\in\partial H$ and a given $w\in H$. Therefore we simply denote the height function by $h$ instead of $h_{(a,w)}$. 

\begin{propo}\label{PropoQuaisLinGeodVert}
Let $\alpha$ be a vertical geodesic of $H$. We have the following control on the height along~$\alpha$:
\begin{equation}
\forall t_1,t_2\in\mathbb{R},\ t_2-t_1-200\delta \leq h\bigl( \alpha(t_2)\bigr)-h\big(\alpha(t_1)\big)\leq  t_2-t_1+200\delta.\nonumber
\end{equation}
\end{propo}

\begin{proof}\ 
Let $t_1,t_2\in\mathbb{R}$, then:
\begin{align}
h( \alpha(t_2))-h(\alpha(t_1))&=\beta\big( \alpha(t_1),w\big)-\beta\big(\alpha(t_2),w\big)\nonumber
\\&= \beta\big( \alpha(t_1),\alpha(t_2)\big)-\Big(\beta\big( \alpha(t_2),w\big)-\beta\big(\alpha(t_1),w\big)+\beta\big( \alpha(t_1),\alpha(t_2)\big)\Big).\nonumber
\end{align}
The third point of Proposition \ref{PropHautGDLH} applied to the last bracket gives:
\begin{equation}\label{formule1}
\beta\big( \alpha(t_1),\alpha(t_2)\big)-200\delta\leq h( \alpha(t_2))-h(\alpha(t_1))\leq \beta\big( \alpha(t_1),\alpha(t_2)\big)+200\delta.
\end{equation}
Since $t\mapsto \alpha(t+t_2)$ is a vertical geodesic starting at $\alpha(t_2)$ we have:
\begin{align}
\beta\big( \alpha(t_1),\alpha(t_2)\big)&=\sup\Big\lbrace \limsup\limits_{t\rightarrow + \infty}\big(d(\alpha(t_1),k(t))-t\big)\Big\mid \text{k vertical rays starting at } \alpha(t_2) \Big\rbrace\nonumber
\\&\geq\limsup\limits_{t\rightarrow + \infty}\Big(d\big(\alpha(t_1),\alpha(t+t_2)\big)-t\Big)\nonumber
\\&\geq\limsup\limits_{t\rightarrow + \infty}\big(|t+t_2-t_1|-t\big)\geq t_2-t_1\nonumber \text{, for t large enough.}\nonumber
\end{align}
Using this last inequality in inequality (\ref{formule1}) we get $t_2-t_1-200\delta\leq h( \alpha(t_2))-h(\alpha(t_1))$. The result follows by exchanging the roles of $t_1$ and $t_2$.
\end{proof}

Using Proposition \ref{PropoQuaisLinGeodVert} with $t_1=0$ and $t_2=t$, the next corollary holds.

\begin{cor}\label{CoroGeodVertQuasiLin}
Let $\alpha$ be a vertical geodesic parametrised by arclength and such that $h(\alpha(0))=0$. We have:
\begin{equation}
\forall t\in\mathbb{R},\ |h(\alpha(t))-t| \leq 200\delta.\nonumber
\end{equation}
\end{cor}

From now on, $H$ will be a proper, geodesic, Gromov hyperbolic, Busemann space. Hence the height function is convex along a vertical geodesic.

\begin{propr}[Prop. $12.1.5$ in p.263 of Papadopoulos \cite{Papa}]\label{PropConvBuseFonction}
Let $\delta\geq 0$ be a non negative number. Let $H$ be a proper $\delta$-hyperbolic, Busemann space. For every geodesic $\alpha$, the function $t\mapsto -h(\alpha(t))$ is convex.
\end{propr}

The Busemann hypothesis implies that the height along geodesic behaves nicely. This means that we can drop the constant $200\delta$ from Corollary \ref{CoroGeodVertQuasiLin}. It is one of the main reasons why we require our spaces to be Busemann spaces.  

\begin{propo}\label{PropoHautLin}
Let $H$ be a $\delta$-hyperbolic and Busemann space and let $V:\mathbb{R}\to H$ be a path of $H$. Then $V$ is a vertical geodesic if and only if $\exists c\in\mathbb{R}$ such that $\forall t\in\mathbb{R},\ h(V(t))=t+c$.
\end{propo}

\begin{proof} Let $V$ be a vertical geodesic in $H$. By Property \ref{PropConvBuseFonction} we have that $t\mapsto -h(V(t))$ is convex. Furthermore, from Corollary \ref{CoroGeodVertQuasiLin}, we get $|h(V(t))-t|\leq 200\delta$. Thereby the bounded convex function $t\mapsto t-h(V(t))$ is constant. Then there exists a real number $c$ such that $\forall t\in\mathbb{R},\ h(V(t))=t+c$.
\\We now assume that there exists a real number $c$ such that $\forall t\in\mathbb{R},\ h(V(t))=t+c$. Therefore, for all real numbers $t_1$ and $t_2$ we have $d\big(V(t_1),V(t_2)\big)\geq \Delta h\big(V(t_1),V(t_2)\big)=|t_1-t_2|$. By definition $V$ is a connected path, hence $|t_1-t_2|\geq d\big(V(t_1),V(t_2)\big)$ which implies with the previous sentence that $|t_1-t_2|=d\big(V(t_1),V(t_2)\big)$, then $V$ is a geodesic. Furthermore $\lim\limits_{t\rightarrow+\infty}h(V(t))=+\infty$, which implies by definition that $V$ is a vertical geodesic. 
\end{proof}

A metric space is called geodesically complete if all its geodesic segments can be prolonged into geodesic lines. In $H$ is geodesically complete in addition to its other assumptions, then any point of $H$ is included in a vertical geodesic line. 

\begin{propr}\label{ExistsVertGeodInx}
Let $H$ be a $\delta$-hyperbolic Busemann geodesically complete space. Then for all $x\in H$ there exists a vertical geodesic $V_x:\mathbb{R}\rightarrow H$ such that $V_x$ contains $x$  
\end{propr}
\begin{proof}
Let us consider in this proof $w\in H$ and $a\in\partial H$, from which we constructed the height $h$ of our space $H$. Then by definition we have $h_{(a,w)}=h$. Proposition 12.2.4 of \cite{Papa} ensures the existence of a geodesic ray $R_x\in a$ starting at $x$. Furthermore as $H$ is geodesically complete $R_x$ can be prolonged into a geodesic $V_x:\mathbb{R}\rightarrow H$ such that $V_x([0;+\infty[)\in a$, hence $V_x$ is a vertical geodesic.
\end{proof}

\section{Horospherical products}\label{SectHoroDefn}

In this part we generalise the definition of horospherical product, as seen in \cite{EFW1} for two trees or two hyperbolic planes, to any pair of proper, geodesically complete, Gromov hyperbolic, Busemann spaces. We recall that given a proper, $\delta$-hyperbolic space $H$ with distinguished $a\in\partial H$ and $w\in H$, we defined the height function on $H$ in Section \ref{SecHeight} from the Busemann functions with respect to $a$ and $w$.

\subsection{Definitions}

Let $X$ and $Y$ be two $\delta-$hyperbolic spaces. We fix the base points $w_X\in X,\ w_Y\in Y$ and the directions in the boundaries $a_X\in\partial X,\ a_Y\in\partial Y$. We consider their heights functions $h_X$ and $h_Y$ respectively on $X$ and $Y$.

\begin{defn}[Horospherical product]\label{DefHoro}
The horospherical product of $X$ and $Y$, denoted by $X\bowtie Y= X\bowtie Y $ is
\begin{equation}
 X\bowtie Y :=\big\lbrace (p_X,p_Y)\in X\times Y\ | h_X (p_X)+h_Y (p_Y)=0 \big\rbrace.\nonumber
\end{equation}
\end{defn}

From now on, with slight abuse, we omit the base points and fixed points on the boundary in the construction of the horospherical product. The metric space $ X\bowtie Y $ refers to a horospherical product of two Gromov hyperbolic Busemann spaces.  We choose to denote $X$ and $Y$ the two components in order to identify easily which objects are in which component. In order to define a Horospherical product in a wider settings, one might only a Busemann function on a metric space.
\\One of our goals is to understand the shape of geodesics in $ X\bowtie Y $ according to a given distance on it. In a cartesian product the chosen distance changes the behaviour of geodesics. However we show that in a horopsherical product the shape of geodesics does not change for a large family of distances, up to an additive constant. 
\\\\We will define the distances on $X\bowtie Y= X\bowtie Y $ as length path metrics induced by distances on $X\times Y$. A lot of natural distances on the cartesian product $X\times Y$ come from norms on the vector space $\mathbb{R}^2$. Let $N$ be such a norm and let us denote $d_N:=N(d_{X},d_{Y})$, which means that for all couples $(p_X,p_Y),(q_X,q_Y)\in X\times Y$ we have that $d_N\big((p_X,p_Y),(q_X,q_Y)\big)=N\big(d_{X}(p_X,q_X),d_{Y}(p_Y,q_Y)\big)$. The length $l_N(\gamma)$ of a path $\gamma=(\gamma_X,\gamma_Y)$ in the metric space $\Big(X\times Y,d_N\Big)$ is defined by:
\begin{equation}
l_N(\gamma)=\sup\limits_{\theta\in\Theta([t_1,t_2])}\left(\sum\limits_{i=1}^{n_{\theta}-1}d_{N}(\gamma(\theta_i),\gamma(\theta_{i+1}))  \right).\nonumber
\end{equation}
Where $\Theta([t_1,t_2])$ is the set of subdivisions of $[t_1,t_2]$. Then the $N$-path metrics on $ X\bowtie Y $ is:

\begin{defn}[The $N$-path metrics on $ X\bowtie Y $]\label{DefNDistHoro}
Let $N$ be a norm on the vector space $\mathbb{R}^2$. The $N$-path metric on $X\bowtie Y$, denoted by $d_{\bowtie}$, is the length path metric induced by the distance $N(d_{X},d_{Y})$ on $X \times Y$. For all $p$ and $q$ in $ X\bowtie Y $ we have:
\begin{align}
d_{\bowtie}(p,q)=\inf\lbrace l_N(\gamma)\vert \gamma\ path\ in\  X\bowtie Y \ linking\ p\ to\ q\rbrace.
\end{align}
\end{defn}

Any norm $N$ on $\mathbb{R}^2$ can be normalised such that $N(1,1)=1$. We call admissible any such norm which satisfies an additional condition. 

\begin{defn}[Admissible norm]\label{DefDistHoro}
Let $N$ be a norm on the vector space $\mathbb{R}^2$ such that $N(1,1)=1$. The norm $N$ is called admissible if and only if for all real $a$ and $b$ we have:
\begin{equation}
N(a,b)\geq \frac{a+b}{2}.
\end{equation}
Since all norms are equivalent in $\mathbb{R}^2$, there exists a constant $C_N\geq 1$ such that:
\begin{equation}
N(a,b)\leq C_N\frac{a+b}{2}.
\end{equation}
\end{defn}

As an example, any $l_p$ norm with $p\geq 1$ is admissible.

\begin{propr}\label{ProprSplitLenght}
Let $N$ be an admissible norm on the vector space $\mathbb{R}^2$. Let $\gamma:=(\gamma_X,\gamma_Y)\subset X\times Y$ be a connected path. Then we have:
\begin{equation}
\frac{l_{X}(\gamma_X)+l_{Y}(\gamma_Y)}{2}\leq l_N(\gamma)\leq C_N\frac{l_{X}(\gamma_X)+l_{Y}(\gamma_Y)}{2}.\nonumber
\end{equation}
\end{propr}

\begin{proof}
Let $\gamma:=(\gamma_X,\gamma_Y):[t_1,t_2]\rightarrow X\times Y$ be a connected path and $\theta$ a subdivision of $[t_1,t_2]$, then by the definition of the length:
\begin{align*}
l_N(\gamma)&\geq\sum\limits_{i=1}^{n_{\theta}-1}d_{N}(\gamma(\theta_i),\gamma(\theta_{i+1}))=\sum\limits_{i=1}^{n_{\theta}-1}N\Big(d_{X}\big(\gamma_X(\theta_i),\gamma_X(\theta_{i+1})\big),d_{ Y}\big(\gamma_Y(\theta_i),\gamma_Y(\theta_{i+1})\big)\Big)
\\&\geq\sum\limits_{i=1}^{n_{\theta}-1}\frac{1}{2}\Big(d_{X}\big(\gamma_X(\theta_i),\gamma_X(\theta_{i+1})\big)+d_{ Y}\big(\gamma_Y(\theta_i),\gamma_Y(\theta_{i+1})\big)\Big),\ \text{since }N\text{ is admissible}.
\\&\geq\frac{1}{2}\left(\sum\limits_{i=1}^{n_{\theta}-1}d_{X}\big(\gamma_X(\theta_i),\gamma_X(\theta_{i+1})\big)+\sum\limits_{i=1}^{n_{\theta}-1}d_{ Y}\big(\gamma_Y(\theta_i),\gamma_Y(\theta_{i+1})\big)\right).
\end{align*}
Any couple of subdivision $\theta_1$ and $\theta_2$ can be merge into a subdivision $\theta$ that contains $\theta_1$ and $\theta_2$. Furthermore the last inequality holds for any subdivision $\theta$, hence by taking the supremum on all the subdivisions we have:
\begin{equation}
l_N(\gamma)\geq \frac{l_{X}(\gamma_X)+l_{Y}(\gamma_Y)}{2}.\nonumber
\end{equation}
Furthermore, we have that $\forall a,b\in\mathbb{R}$, $N(a,b)\leq C_N\frac{a+b}{2}$, hence:
\begin{align*}
\sum\limits_{i=1}^{n_{\theta}-1}d_{N}(\gamma(\theta_i),\gamma(\theta_{i+1}))&\leq \frac{C_N}{2}\left(\sum\limits_{i=1}^{n_{\theta}-1}d_{X}(\gamma_X(\theta_i),\gamma(\theta_{i+1}))+\sum\limits_{i=1}^{n_{\theta}-1}d_{Y}(\gamma_Y(\theta_i),\gamma_Y(\theta_{i+1}))\right)
\\&\leq C_N\frac{l_{X}(\gamma_X)+l_{X}(\gamma_X)}{2}
\end{align*} 
Since last inequality holds for any subdivision $\theta$, we have that $l_N(\gamma)\leq C_N\frac{l_{X}(\gamma_X)+l_{X}(\gamma_X)}{2}$.

\end{proof}

The definition of height on $X$ and $Y$ is used to construct a height function on $X\bowtie Y$.
 
\begin{defn}[Height on $ X\bowtie Y $]
The height $h(p)$ of a point $p=(p_X,p_Y)\in X\bowtie Y$ is defined as $h(p)=h_X(p_X)=-h_Y(p_Y)$.
\end{defn}

On Gromov hyperbolic spaces we have that de distance between two points is greater than their height difference. The same occurs on horospherical products given with an admissible norm. Let $x$ and $y$ be two points of $ X\bowtie Y $, and let us denote $\Delta h(p,q) := |h(p)-h(q)|$ their height difference.

\begin{lemma}\label{LemDistBigHaut}
Let $N$ be an admissible norm, and let $d_{\bowtie}$ the distance on $X\bowtie Y$ induced by $N$. Then the height function is $1$-Lipschitz with respect to the distance $d_{\bowtie}$, \textsl{i.e.},
\begin{equation}
\forall p,q \in X\bowtie Y ,\quad d_{\bowtie}(p,q)\geq \Delta h(p,q).
\end{equation}
\end{lemma}

\begin{proof}
Since $N$ is admissible we have:
\begin{align*}
d_{\bowtie}(p,q)&\geq \frac{d_{X}(p_X,q_X)+d_{Y}(p_Y,q_Y)}{2}\geq \frac{\Delta h(p_X,q_X)+\Delta h(p_Y,q_Y)}{2}
\\&=\Delta h(p_X,q_X)=\Delta h(p,q).
\end{align*}
\end{proof}

Following Proposition \ref{PropoHautLin}, we define a notion of vertical paths in a horospherical product.
 
\begin{defn}[Vertical paths in $ X\bowtie Y $]\label{DefVertGeodInHoro}
Let $V:\mathbb{R}\to  X\bowtie Y $ be a connected path. We say that $V$ is vertical if and only if there exists a parametrisation by arclength of $V$ such that $h(V(t))=t$ for all $t$.
\end{defn}

Actually, a vertical path of a horospherical product is a geodesic.

\begin{lemma}
Let $N$ be an admissible norm. Let $V:\mathbb{R}\to  X\bowtie Y $ be a vertical path. Then $V$ is a geodesic of $( X\bowtie Y ,d_{\bowtie})$.
\end{lemma}
\begin{proof}
Let $t_1,t_2\in\mathbb{R}$. The path $V$ is vertical therefore $\Delta h\big(V(t_1),V(t_2)\big)=|t_1-t_2|$. Since $V$ is connected and parametrised by arclength, we have that:
\begin{align*}
|t_1-t_2|=l_N\left(V_{|[t_1,t_2]}\right)&\geq d_{\bowtie}\big(V(t_1),V(t_2)\big)
\\&\geq  \Delta h\big(V(t_1),V(t_2)\big)=|t_1-t_2|.
\end{align*}
Then $d_{\bowtie}\big(V(t_1),V(t_2)\big)=|t_1-t_2|$, which ends the proof.
\end{proof}

Such geodesics are called vertical geodesics. Next proposition tells us that vertical geodesics of $X\bowtie Y$ are exactly couples of vertical geodesics of $X$ and $Y$.

\begin{propo}\label{PropGeodVertInHoro}
Let $N$ be an admissible norm and let $V=(V_X,V_Y):\mathbb{R}\to  X\bowtie Y $ be a geodesic of $( X\bowtie Y ,d_{\bowtie})$. The two following properties are equivalent:
\begin{enumerate}
\item $V$ is a vertical geodesic of $( X\bowtie Y ,d_{\bowtie})$
\item $V_X$ and $V_Y$ are respectively vertical geodesics of $X$ and $Y$.
\end{enumerate}
\end{propo}
\begin{proof}\ 
Let us first assume that $V$ be a vertical geodesic, we have for all real $t$ that $h(V_X(t))=h(V(t))=t$, hence $\forall t_1,t_2\in\mathbb{R}$:
\begin{align}
d_{X}\big(V_X(t_1),V_X(t_2)\big)\geq \Delta h \big(V_X(t_1),V_X(t_2)\big)=|t_1-t_2|.\label{lowboundprojgeod}
\end{align}
Similarly we have that $d_{Y}\big(V_Y(t_1),V_Y(t_2)\big)\geq |t_1-t_2|$. Using that $N$ is admissible and that $V$ is a geodesic we have:
\begin{align*}
d_{X}\big(V_X(t_1),V_X(t_2)\big)&=2\frac{d_{X}\big(V_X(t_1),V_X(t_2)\big)+d_{Y}\big(V_Y(t_1),V_Y(t_2)\big)}{2}-d_{Y}\big(V_Y(t_1),V_Y(t_2)\big)
\\&\leq 2d_{\bowtie}\big(V(t_1),V(t_2)\big)-|t_1-t_2|=|t_1-t_2|.
\end{align*}
Combine with inequality (\ref{lowboundprojgeod}) we have that $d_{X}\big(V_X(t_1),V_X(t_2)\big)=|t_1-t_2|$, hence $V_X$ is a vertical geodesic of $X$. Similarly, $V_Y$ is a vertical geodesic $Y$.
\\Let us assume that $V_X$ and $V_Y$ are vertical geodesics of $X$ and $Y$. Let $t_1,t_2\in\mathbb{R}$, we have:
\begin{align*}
d_{\bowtie}(V(t_1),V(t_2))&=\sup\limits_{\theta\in\Theta ([t_1,t_2])} \left( \sum\limits_{i=1}^{n_{\theta}-1}d_{N}(V(\theta_i),V(\theta_{i+1}))\right)
\\&=\sup\limits_{\theta\in\Theta ([t_1,t_2])} \left( \sum\limits_{i=1}^{n_{\theta}-1}N\Big(d_{X}\big(V_X(\theta_i),V_X(\theta_{i+1})\big),d_{Y}\big(V_Y(\theta_i),V_Y(\theta_{i+1})\big)\Big)\right)
\\&=\sup\limits_{\theta\in\Theta ([t_1,t_2])} \left( \sum\limits_{i=1}^{n_{\theta}-1}N\Big(\Delta h\big(V_X(\theta_i),V_X(\theta_{i+1})\big),\Delta h\big(V_Y(\theta_i),V_Y(\theta_{i+1})\big)\Big)\right)
\\&=\sup\limits_{\theta\in\Theta ([t_1,t_2])} \left( N(1,1)\sum\limits_{i=1}^{n_{\theta}-1}\Delta h\big(V_X(\theta_i),V_X(\theta_{i+1})\big)\right)
\\&=N(1,1)\Delta h\big(V_X(t_1),V_X(t_2)\big)=|t_1-t_2|,\text{ since }N(1,1)=1.
\end{align*}
Where $\Theta ([t_1,t_2])$ is the set of subdivision of $[t_1,t_2]$. Hence the proposition is proved.
\end{proof}
This previous result is the main reason why we are working with distances which came from admissible norms.

\begin{defn}
A geodesic ray of $X\bowtie Y$ is called vertical if it is a subset of a vertical geodesic.
\end{defn}

A metric space is called geodesically complete if all its geodesic segments can be prolonged into geodesic lines. If $X$ and $Y$ are proper hyperbolic geodesically complete Busemann spaces, their horospherical product $ X\bowtie Y $ is connected.

\begin{propr}\label{HoroProdConnected}
Let $X$ and $Y$ be two proper, geodesically complete, $\delta$-hyperbolic, Busemann spaces. Let $X\bowtie Y$ be their horospherical product. Then $ X\bowtie Y $ is connected, furthermore $\frac{1}{2}(d_{X}+d_{Y})\leq d_{ X\bowtie Y }\leq 2C_N(d_{X}+d_{Y})$.
\end{propr}

\begin{proof}
Let $p=(p_X,p_Y)$ and $q=(q_X,q_Y)$ be two points of $ X\bowtie Y $. From Property \ref{ExistsVertGeodInx}, there exists a vertical geodesic $V_{p_Y}$ such that $p_Y$ is in the image of $V_{p_Y}$, and there exists a vertical geodesic $V_{q_X}$ such that $q_X$ is in the image of $V_{q_X}$. Let $q_Y'$ be the point of $V_{p_Y}$ at height $h(q_Y)$. Let $\alpha_X$ be a geodesic of $X$ linking $p_X$ to $q_X$ and let $\alpha_Y'$ be a geodesic of $Y$ linking $q_Y'$ to $q_Y$. We will connect $x$ to $y$ with a path composed with pieces of $\alpha_X$, $\alpha_Y'$, $V_{p_Y}$ and $V_{q_X}$. 
\\We first link $(p_X,p_Y)$ to $(q_X,q_Y')$ with $\alpha_X$ and $V_{p_Y}$. It is possible since $V_{p_Y}$ is parametrised by its height. More precisely we construct the following path $c_1$:
\begin{align*}
\forall t\in[0,d(p_X,q_X)],\ c_1(t)=\Big(\alpha_X(t), V_{p_Y}\big( -h(\alpha_X(t))\big)\Big).
\end{align*}
Since $V_{p_Y}$ is parametrised by its height, we have $h\left( V_{p_Y}\big( -h(\alpha_X(t))\big)\right)=-h(\alpha_X(t))$ which implies $c_1(t)\in X\bowtie Y $. Furthermore, using the fact that the height is 1-Lipschitz, we have $\forall t_1,t_2\in[0,d(p_X,q_X)]$:
\begin{align*}
d_{Y}\Big(V_{p_Y}\big( -h(\alpha_X(t_1))\big),V_{p_Y}\big( -h(\alpha_X(t_2))\big)\Big)=| h(\alpha_X(t_1))-h(\alpha_X(t_2))|\leq d_{X}(\alpha_X(t_1),\alpha_X(t_2)).
\end{align*}
Hence $c_{1,Y}:t\mapsto V_{p_Y}\big( -h(\alpha_X(t))\big)$ is a connected path such that $l(c_{1,Y})\leq l(\alpha_X)\leq d_{X}(p_X,q_X)$. Hence $c_1$ is a connected path linking $(p_X,p_Y)$ to $(q_X,q_Y')$. Using 
Property \ref{ProprSplitLenght} on $c_1$ provides us with:
\begin{align*}
l_N(c_1)&\leq \frac{C_N}{2} (l(c_{1,Y})+l(\alpha_X))\leq C_N l(\alpha_X)
\\&\leq C_N d_{X}(p_X,q_X)
\end{align*}
We recall that by definition $q_Y'=V_{p_Y}(h(q_Y))$.
We show similarly that $c_{2}:t\mapsto \Big(V_{q_X}\big( -h(\alpha_Y'(t))\big),\alpha_Y'(t) \Big)$ is a connected path linking $(q_X,q_Y')$ to $(q_X,q_Y)$ such that:
\begin{align*}
l(c_2)&\leq C_N d_{Y}(q_Y',q_Y)\leq C_N\big(d_{Y}(q_Y',p_Y)+d_{Y}(p_Y,q_Y)\big)
\\&= C_N\big(\Delta h(p_Y,q_Y)+d_{Y}(p_Y,q_Y)\big)\text{, since }q_Y'=V_{p_Y}(h(q_Y))
\\&\leq 2C_N d_{Y}(p_Y,q_Y).
\end{align*}
Hence, there exists a connected path $c=c_1\cup c_2$ linking $p$ to $q$ such that:
\begin{equation}
l(c)\leq C_N d_{X}(p_X,q_X)+2C_N d_{Y}(p_Y,q_Y)\leq  2C_N\big(d_{X}(p_X,q_X)+d_{Y}(p_Y,q_Y)\big).
\end{equation} 
\end{proof}

However if the two components $X$ and $Y$ are not geodesically complete, $ X\bowtie Y $ may not be connected. 

\begin{example}
Let $X$ and $Y$ be two graphs, constructed from an infinite line $\mathbb{Z}$ (indexed by $\mathbb{Z}$) with an additional vertex glued on the $0$ for $X$ and on the $-2$ for $Y$. Their construction are illustrated in Figure \ref{FigUnconnectedHoroProd}. They are two 0-hyperbolic Busemann spaces which are not geodesically complete. Let $w_{X}\in X$ be the vertex indexed by $0$ in $X$, and let $w_{Y}\in Y$ be the vertex indexed by $-2$ in $Y$. We choose them to be the base points of $X$ and $Y$. Since $\partial X$ and $\partial Y$ contain two points each, we fix in both cases the point of the boundary $a_X$ or $a_Y$ to be the one that contains the geodesic ray indexed by $\mathbb{N}$. On figure \ref{FigUnconnectedHoroProd}, we denoted the height of a vertex inside this one. Then the horospherical product $X\bowtie Y$ taken with the $\ell_1$ path metric is not connected. Since some vertices of $X$ and $Y$ are not contained in a vertical geodesic, one may not be able to adapt its height correctly while constructing a path joining $\left(p^X_{-1},p^Y_{(2,1)}\right)$ to $\left(p^X_{(0,-1)},p^Y_{(2,1)}\right)$.
\end{example}

\begin{figure}
\includegraphics[scale=0.8]{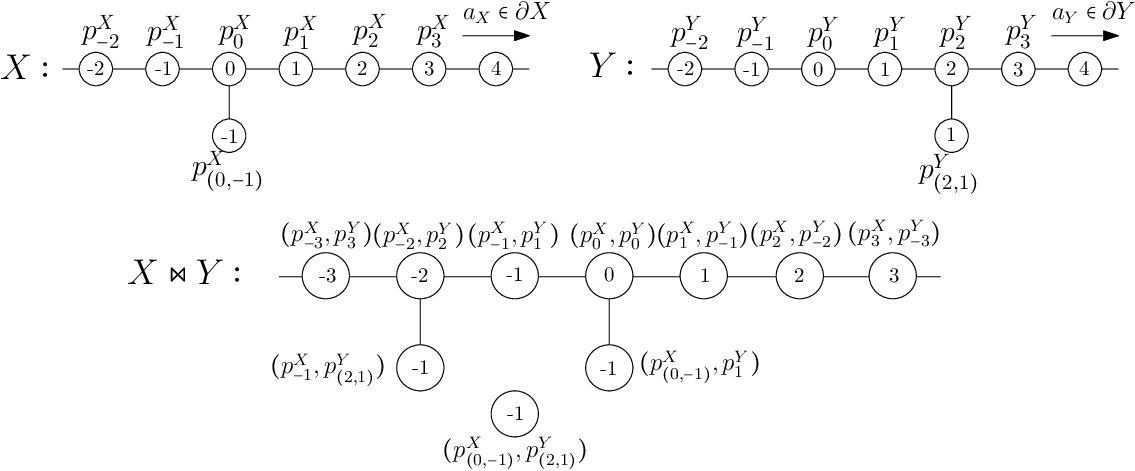} 
\centering
\caption{Example of horospherical product which is not connected. The number in a vertex is the height of that vertex.}
\label{FigUnconnectedHoroProd}
\end{figure}

It is not clear that a horospherical product is still connected without the hypothesis that $X$ and $Y$ are Busemann spaces. In that case we would need a "coarse" definition of horospherical product. Indeed, the height along geodesics would not be smooth as in Proposition \ref{PropoHautLin}, therefore the condition requiring to have two exact opposite heights would not suits.

\subsection{Examples}

A \textbf{Heintze group} is a Lie group of the form $\mathbb{R}\ltimes_A N$ defined by the action on $\mathbb{R}$, $t\mapsto \exp(tA)$, with $N$ a simply connected nilpotent Lie group and with $A\in \mathrm{Lie}(A)$ a derivation whose eigenvalues have positive real parts. Heintze proved in \cite{Heintze} that any simply connected, negatively curved Lie group is isomorphic to a Heintze group.

Moreover, a Busemann metric space is simply connected, hence any Gromov hyperbolic, Busemann Lie group is isomorphic to a Heintze group. Consequently, Heintze groups are natural candidates for the two components from which a horospherical product is constructed. In his paper \cite{Xie}, Xie classifies the subfamily of all negatively curved Lie groups $\mathbb{R}\ltimes \mathbb{R}^n$ up to quasi-isometry.

Let $H_1:=\mathbb{R}\ltimes_{A_1}N_1$ and $H_2:=\mathbb{R}\ltimes_{A_2}N_2$ be two Heintze groups, then $H_1 \bowtie H_2$ is isomorphic to $\mathbb{R}\ltimes_{\mathrm{Diag}(A_1,-A_2)} (N_1\times N_2)$, where $\mathrm{Diag}(A_1,-A_2)$ is the block diagonal matrix containing $A_1$ and $-A_2$ on its diagonal. In fact, We have that $H_1 \times H_2$ is the group $\mathbb{R}^2\ltimes_{(A_1,A_2)} (N_1\times N_2)$ defined by the action on $\mathbb{R}^2$, $(t_1,t_2)\mapsto (\exp(t_1 A_1),\exp(t_2 A_2))$. Let $(0,e_{N_1})\in N_1$, $(0,e_{N_2})\in N_2$ be the two base points, and let $t\mapsto (t,e_{N_1})$ and  $t\mapsto (t,e_{N_2})$ be there respective vertical geodesic rays corresponding to the chosen Busemann functions. Then we have that for all $(t,n)\in H_i$, $h(t,n)=t$. Under this setting we have that 
\begin{align*}
H_1\bowtie H_2 &= \left\lbrace(t_1,t_2,n_1,n_2)\in H_1\times H_2\mid t_1=-t_2\right\rbrace= \left\lbrace(t,-t,n_1,n_2)\in H_1\times H_2\right\rbrace.
\end{align*}
Thanks to this characterisation, we show that $H_1\bowtie H_2$ is a subgroup of $\mathbb{R}^2\ltimes_{(A_1,A_2)} (N_1\times N_2)$. Furthermore the following map is an isomorphism
\begin{align*}
H_1\bowtie H_2 &\to \mathbb{R}\ltimes_{\mathrm{Diag}(A_1,-A_2)} (N_1\times N_2)
\\(t,-t,n_1,n_2)&\mapsto (t,n_1,n_2),
\end{align*}
where  $\mathbb{R}\ltimes_{\mathrm{Diag}(A_1,-A_2)} (N_1\times N_2)$ is determined by the action $t\mapsto (\exp(t A_1),\exp(-t A_2))$. Therefore, we have that
\begin{align*}
(\mathbb{R}\ltimes_{A_1}N_1)\bowtie (\mathbb{R}\ltimes_{A_2}N_2) \cong_{iso} \mathbb{R}\ltimes_{\mathrm{Diag}(A_1,-A_2)} (N_1\times N_2)
\end{align*}
The Sol geometries are specific cases of such solvable Lie groups when $N_i=\mathbb{R}$ for $i\in \{1,2\}$, and where the matrices $A_i$ are positive reals. In this context, for $m>0$ we have that $\mathbb{R}\ltimes_m \mathbb{R}$ is the Log model of a real hyperbolic plan, otherwise stated the Riemannian manifold with coordinates $(x,z)\in\mathbb{R}^2$ endowed with the Riemannian metric $ds^2=e^{-2mz}dx^2+dz^2$. Then $(\mathbb{R}\ltimes_m \mathbb{R})\bowtie (\mathbb{R}\ltimes_n \mathbb{R}) = \mathbb{R}\ltimes_{\mathrm{Diag}(m,-n)} \mathbb{R}^{2}$ is a Sol geometry, or also the Riemannian manifold with coordinates $(x_1,x_2,z)\in\mathbb{R}^3$ endowed with the Riemannian metric
\begin{equation}
ds^2=e^{-2mz}dx_1^2+e^{2nz}dx_2^2+dz^2.\nonumber
\end{equation}
%The standard Sol geometry can also seen as the Riemannian manifold with coordinates $(x,y,z)\in\mathbb{R}^3$ and with the Riemannian metric $ds^2=dz^2+e^{2z}dx^2+e^{-2z}dy^2$. It is the horospherical product of two hyperbolic planes $\mathbb{H}^2$ defined in there Log model as the Riemannian manifold with coordinates $(x,z)\in\mathbb{R}^2$ and with the Riemannian metric $ds^2=dz^2+e^{-2z}dx^2$. We fix $w=(0,0)$ as the base point of $\mathbb{H}$ and the "upward" direction $a$ as the point on the boundary. In that case the height function in regards to $(a,w)$ taken on a point $(x,z)\in\mathbb{H}$ is $h_{(a,w)}(x,z)=z$. We now look at the horospherical product $\mathbb{H}^2\bowtie\mathbb{H}^2:=\lbrace (x_1,z_1,x_2,z_2)\in\mathbb{R}^2\times\mathbb{R}^2 \vert z_1=-z_2\rbrace$ taken with the $\ell_2$ path metric. Since the second and the fourth variable are exactly opposite, we merge them into one. Hence we have that $\mathbb{H}^2\bowtie\mathbb{H}^2$ is isometric to the space $\lbrace (x_1,x_2,z_1)\in\mathbb{R}^3\rbrace$ with the metric
%\begin{equation}
%ds^2=dz_1^2+e^{-2z_1}dx_1^2+dz_1^2+e^{2z_1}dx_2^2=2dz_1^2+e^{-2z_1}dx_1^2+e^{2z_1}dx_2^2.\nonumber
%\end{equation}
%Changing the coordinates by dividing $x_1$ and $x_2$ by $2$ tells us that this space is isometric to Sol.
A first discrete example of horospherical product is the family of Diestel-Leader graphs defined by $DL(n,m)=T_n\bowtie T_m$ with $n,m\geq 2$ and where $T_n$ and $T_m$ are regular trees. We see $T_n$ and $T_m$ as connected metric spaces with the usual distance on them. By choosing half of the $\ell_1$ path metric on $DL(n,m)$, this horospherical product becomes a graph with the natural distance on it. Indeed, the set of vertices of $DL(n,m)$ is then defined by the subset of couples of vertices of $T_n\times T_m$ included in $DL(n,m)$. In this horospherical product, two points $(p_n,p_m)$ and $(q_n,q_m)$ of $DL(n,m)$ are connected by an edge if and only if $p_n$ and $q_n$ are connected by an edge in $T_n$ and if $p_m$ and $q_m$ are connected by an edge in $T_m$. Furthermore, when $n=m$, there is a one-to-one correspondence between $DL(n,n)$ and the Cayley graph of the lamplighter group $\mathbb{Z}_Y\wr \mathbb{Z}$, see \cite{Woess2} for further details.
\begin{figure}[h!]
    \centering
    \begin{minipage}{0.45\textwidth}
        \centering
        \includegraphics[scale=0.90]{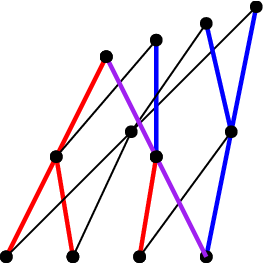}
        \caption{A portion of the graph $\mathbb{T}_3\bowtie\mathbb{T}_3$ }
    \end{minipage}\hfill
    \begin{minipage}{0.45\textwidth}
        \centering
        \includegraphics[scale=0.7]{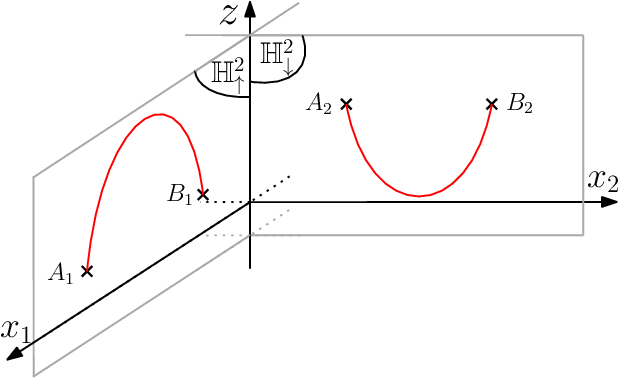} 
        \caption{The Sol geometry and two geodesics of embedded copies of $\mathbb{H}^2$}
    \end{minipage}
\end{figure}

Depending on the case, we either used the $\ell_1$ path metric or the $\ell_2$ path metric. However, we will see in Proposition \ref{l1samel2} that it does not matter, up to an additive uniform constant. Quasi-isometric rigidity results in the Diestel-Leader graphs and the Sol geometry have been proved using the same techniques in \cite{EFW1} and \cite{EFW2}.

The horospherical product of a hyperbolic plane and a regular tree has been studied as the 2-complex of Baumslag-Solitar groups in \cite{Treebol}, they are called the treebolic spaces. The distance they choose on the treebolic spaces is similar to ours. In fact our Proposition \ref{lengthGeod} and their Proposition $2.8$ page 9 (in \cite{Treebol}) tell us they are equal up to an additive constant. Rigidity results on the quasi-isometry classification of the treebolic spaces were brought up in \cite{FB1} and \cite{FB2}.

\section{Estimates on the length of specific paths}\label{SecMetricEstimat}
\subsection{Geodesics in Gromov hyperbolic Busemann spaces}

This section focuses on length estimates in Gromov hyperbolic Busemann spaces. The central result is Proposition \ref{ExpLengthWhenBelowAndReachPoint}, which presents a lower bound on the length of a path staying between two horospheres. Before moving to the technical results of this section, let us introduce some notations.

\begin{nota}
Unless otherwise specified, $H$ will be a Gromov hyperbolic Busemann geodesically complete proper space. Let $\gamma:I\rightarrow H$ be a connected path. Let us denote the maximal height and the minimal height of this path as follows:
\begin{align*}
h^+(\gamma)=\sup\limits_{t\in I}\big\lbrace h(\gamma(t))\big\rbrace\quad;\quad h^-(\gamma)=\inf\limits_{t\in I}\big\lbrace h(\gamma(t))\big\rbrace.
\end{align*}
Let $x$ and $y$ be two points of $H$, we denote the height difference between them by:
\begin{equation}
\Delta h(x,y)=|h(x)-h(y)|.\nonumber
\end{equation}
We define the relative distance between two points $x$ and $y$ of $H$ as:
\begin{equation}
d_r(x,y)=d(x,y)-\Delta h(x,y).\nonumber
\end{equation}

Let us denote $V_x$ a vertical geodesic containing $x$, we will assume it to be parametrised by arclength. Thanks to Proposition \ref{PropoHautLin} we choose a parametrisation by arclength such that $\forall t\in\mathbb{R},\ h(V_x(t))=t+0$.
\end{nota}

The relative distance between two points quantifies how far a point is from the nearest vertical geodesic containing the other point. 
\\\\In the sequel we want to apply the slim triangles property on ideal triangles, hence we need the following result of \cite{Papa1}.

\begin{propr}[Proposition $2.2$ page $19$ of \cite{Papa1}]\label{ThinTrianglePropr}
Let $a,b$ and $c$ be three points of $X\cup\partial X$. Let $\alpha,\beta,\gamma$ be three geodesics of $X$ linking respectively $b$ to $c$, $c$ to $a$, and $a$ to $b$. Then every point of $\alpha$ is at distance less than $24\delta$ from the union $\beta\cup\gamma$. 
\end{propr}

Next lemma tells us that in order to connect two points, a geodesic needs to go sufficiently high. This height is controlled by the relative distance between these two points. 

\begin{lemma}\label{LEM0}
Let $H$ be a $\delta$-hyperbolic and Busemann metric space, let $x$ and $y$ be two elements of $H$ such that $h(x)\leq h(y)$, and let $\alpha$ be a geodesic linking $x$ to $y$. Let us denote $z=\alpha\left(\Delta h(x,y)+\frac{1}{2}d_r(x,y)\right)$, $x_1:=V_x\left(h(y)+\frac{1}{2}d_r(x,y)\right)$ the point of $V_x$ at height $h(y)+\frac{1}{2}d_r(x,y)$ and $y_1:=V_y\left(h(y)+\frac{1}{2}d_r(x,y)\right)$ the point of $V_y$ at the same height $h(y)+\frac{1}{2}d_r(x,y)$. Then we have:
\begin{enumerate}
\item $h^+(\alpha)\geq h(y)+\frac{1}{2}d_r(x,y)-96\delta$
\item $d\left(z,x_1\right)\leq 144\delta$
\item $d\left(z,y_1\right)\leq 144\delta$
\item $d\left(x_1,y_1\right)\leq 288\delta$.
\end{enumerate}
\end{lemma}

\begin{figure}
\begin{center}
\includegraphics[scale=1.1]{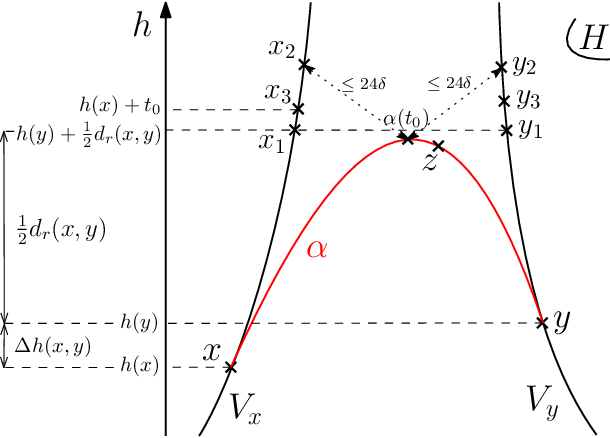} 
\caption{Proof of Lemma \ref{LEM0}}\label{FigLEM0}
\end{center}
\end{figure}

\begin{proof} The lemma and its proof are illustrated in Figure \ref{FigLEM0}. Following Property \ref{ThinTrianglePropr}, the triple of geodesics $\alpha$, $V_x$ and $V_y$ is a $24\delta$-slim triangle. Since the sets $\lbrace t\in[0,d(x,y)]|d(\alpha(t),V_x)\leq 24\delta\rbrace$ and $\lbrace t\in[0,d(x,y)]|d(\alpha(t),V_y)\leq 24\delta\rbrace$ are closed sets covering $[0,d(x,y)]$, their intersection is non empty. Hence there exists $t_0\in[0,d(x,y)]$, $x_2\in V_x$ and $y_2\in V_y$ such that $d(\alpha(t_0),x_2)\leq 24\delta$ and $d(\alpha(t_0),y_2)\leq 24\delta$. Let us first prove that $t_0$ is close to $\Delta h(x,y)+\frac{1}{2}d_r(x,y)$. By the triangle inequality we have that:
\begin{align*}
|t_0-d(x,x_2)|=|d(x,\alpha(t_0))-d(x,x_2)|\leq d(x_2,\alpha(t_0))\leq 24\delta.
\end{align*}
Let us denote $x_3:=V_x(h(x)+t_0)$ the point of $V_x$ at height $h(x)+t_0$, and $y_3=V_y(h(y)+d(x,y)-t_0)$ the point of $V_y$ at height $h(y)+d(x,y)-t_0$. Then by the triangle inequality:
\begin{align}
d(\alpha(t_0),x_3)&\leq d(\alpha(t_0),x_2)+d(x_2,x_3)=d(\alpha(t_0),x_2)+|d(x,x_2)-d(x,x_3)|\nonumber
\\&\leq d(\alpha(t_0),x_2)+|d(x,x_2)-t_0|\leq 48\delta.\label{distprocheentregeod}
\end{align}
In the last inequality we used that $d(x,x_3)=t_0$, which holds by the definition of $x_3$. We show in the same way that $d(\alpha(t_0),y_3)\leq 48\delta$. By the triangle inequality we have $d(x_3,y_3)\leq 96\delta$. As the height function is Lipschitz we have $\Delta h(x_3,y_3)\leq d(x_3,y_3)\leq 96\delta$, which provides us with:
\begin{align}
\left|\frac{1}{2}d_r(x,y)+\Delta h(x,y)-t_0\right|&=\frac{1}{2}\big|d_r(x,y)+\Delta h(x,y)+h(y)-h(x)-2t_0\big|\nonumber
\\&=\frac{1}{2}|h(y)+d(x,y)-t_0-(h(x)+t_0)|=\frac{1}{2}\Delta h(x_3,y_3)\leq \frac{96\delta}{2}\leq 48\delta.\label{t0ettempshautproche}
\end{align}
In particular it gives us that $d(z,\alpha(t_0))\leq 48\delta$. We are now ready to prove the first point using inequalities (\ref{distprocheentregeod}) and (\ref{t0ettempshautproche}):
\begin{align*}
h^+(\alpha)\geq& h(\alpha(t_0))\geq h(x_3)-\Delta h(\alpha(t_0),x_3)\geq h(x)+t_0-48\delta
\\\geq & h(x)+\frac{1}{2}d_r(x,y)+\Delta h(x,y)-96\delta\geq h(y)+\frac{1}{2}d_r(x,y)-96\delta,\text{ as we have }h(x)\leq h(y).
\end{align*}
The second point of our lemma is proved as follows:
\begin{align*}
d(z,x_1)&\leq d(z,\alpha(t_0))+d(\alpha(t_0),x_1)\leq 48\delta+d(\alpha(t_0),x_3)+d(x_3,x_1) 
\\&\leq 96\delta+\left|t_0+h(x)-\left(\frac{1}{2}d_r(x,y)+h(y)\right)\right|=96\delta+\left|t_0-\left(\Delta h(x,y) +\frac{1}{2}d_r(x,y)\right)\right|\leq 144\delta.
\end{align*}
The proof of $3.$ is similar, and $4.$ is obtained from $2.$ and $3.$ by the triangle inequality.
\end{proof}
The next lemma shows that  in the case where $h(x)\leq h(y)$ a geodesic linking $x$ to $y$ is almost vertical until it reaches the height $h(y)$.

\begin{lemma}\label{LinkDrAndDSameHeight}
Let $H$ be a $\delta$-hyperbolic and Busemann space. Let $x$ and $y$ be two points of $H$ such that $h(x)\leq h(y)$.  We define $x':=V_x(h(y))$ to be the point of the vertical geodesic $V_x$ at the same height as $y$. Then:
\begin{equation}
|d_r(x,y)-d(x',y)|\leq 54\delta.
\end{equation}
\end{lemma}

\begin{proof}
Since $H$ is $\delta$-hyperbolic, the geodesic triangle $[x,y]\cup[y,x']\cup[x',x]$ is $\delta$-slim. Then there exists $p_1\in[x,x']$, $p_2\in[x',y]$ and $m\in[x,y]$ such that $d(p_1,m)\leq \delta$ and $d(p_2,m)\leq \delta$. Hence, $h^-([x',y])-\delta\leq h(m)\leq h^+([x,x'])+\delta$. Let $R_{x'}$ and $R_y$ be two vertical geodesic rays respectively contained in $V_x$ and $V_y$ and respectively starting at $x'$ and $y$. Then Property \ref{ThinTrianglePropr} used on the ideal triangle $R_x\cup R_y \cup [x',y]$ implies that $h^-([x',y])\geq h(y)-24\delta$, therefore we have $h^+([x,x'])=h(y)$. Then $h(y)-25\delta\leq h(m)\leq h(y)+\delta$ holds.
It follows that $m$ and $x'$ are close to each other:
\begin{align}
d(m,x')&\leq d(m,p_1)+d(p_1,x')\leq \delta + \Delta h(p_1,x')\leq \delta + \Delta h(p_1,m)+ \Delta h(m,y) +\Delta h(y,x')\nonumber
\\&\leq \delta + d(p_1,m)+ 25\delta + 0\leq 27\delta.\label{EqRandLem1}
\end{align}
Then we give an estimate on the distance between $x$ and $m$:
\begin{equation}
|d(x,m)-\Delta h(x,y)|= |d(x,m)-d(x,x')|\leq d(m,x')\leq 27\delta.\label{EqRandLem2}
\end{equation}
However $d_r(x,y)=d(x,y)-\Delta h(x,y)$ and $d(x,y)=d(x,m)+d(m,y)$, therefore:
\begin{equation}
d_r(x,y)=d(x,m)+d(m,y)-\Delta h(x,y).\label{EqRandLem3}
\end{equation}
Combining inequalities (\ref{EqRandLem2}) and (\ref{EqRandLem3}) we have $|d_r(x,y)-d(m,y)|\leq 27\delta$. Then:
\begin{align*}
|d_r(x,y)-d(x',y)|\leq 27\delta +d(x',m) \leq 54\delta.
\end{align*}
\end{proof}

We are now able to prove the estimates of the next section.

\subsection{Length estimate of paths avoiding horospheres}\label{SecPathAvoidHoro}

Consider a path $\gamma$ and a geodesic $\alpha$ sharing the same end-points in a proper, Gromov hyperbolic, Busemann space. We prove in this section that if the height of $\gamma$ does not reach the maximal height of the geodesic $\alpha$, then $\gamma$ is much longer than $\alpha$. Furthermore, its length increases exponentially with respect to the difference of maximal height between $\gamma$ and $\alpha$. To do so, we make use of Proposition $1.6$ p400 of \cite{BH}, which we recall here. Let us denote by $l(c)$ the length of a path $c$.

\begin{propo}[\cite{BH}]\label{LemmeBrid}
Let $X$ be a $\delta$-hyperbolic geodesic space. Let $c$ be a continuous path in X. If $[p,q]$ is a geodesic segment connecting the endpoints of $c$, then for every $x\in[p,q]$: $$d(x,\mathrm{im}(c))\leq \delta |\log_2 l(c)|+1.$$
\end{propo}

This result implies that a path of $X$ between $p$ and $q$ which avoids the ball of diameter $[p,q]$ has length greater than an exponential of the distance $d(p,q)$. 

From now on we will add as convention that $\delta\geq 1$. For all $\delta_1\leq\delta_2$ a $\delta_1$-slim triangle is also $\delta_2$-slim, hence all $\delta_1$-hyperbolic spaces are $\delta_2$-hyperbolic spaces. That is why we can assume that all Gromov hyperbolic spaces are $\delta$-hyperbolic with $\delta\geq 1$. It allows us to consider $\frac{1}{\delta}$ as a well defined term, we hence avoid the arising of separated cases in some oof the proofs. We also use this assumption to simplify constants appearing in this document. The next result is a similar control on the length of path as Proposition \ref{LemmeBrid}, but we consider that the path is avoiding a horosphere instead of avoiding a ball in~$H$.

\begin{lemma}\label{LemmeAmande}
Let $\delta\geq 1$ and $H$ be a proper, geodesic, $\delta$-hyperbolic, Busemann space. Let $x$ and $y\in H$ and let $V_x$, respectively $V_y$, be a vertical geodesic containing $x$, respectively $y$. Let us consider $t_0\geq\max(h(x),h(y))$ and let us denote $x_0:=V_x(t_0)$ and $y_0:=V_y(t_0)$, the respective points of $V_x$ and $V_y$ at the height $t_0$. Assume that $d(x_0,y_0)> 768\delta$.
\\Then for all connected path $\gamma:[0,T]\rightarrow H$ such that $\gamma(0)=x$, $\gamma(T)=y$ and $h^+(\gamma)\leq h(x_0)$ we have:
\begin{equation}
l(\gamma)\geq\Delta h(x,x_0)+\Delta h(y,y_0) +2^{-386}2^{\frac{1}{2\delta}d(x_0,y_0)}-24\delta .
\end{equation}
\end{lemma}
\begin{figure}
\includegraphics[scale=0.35]{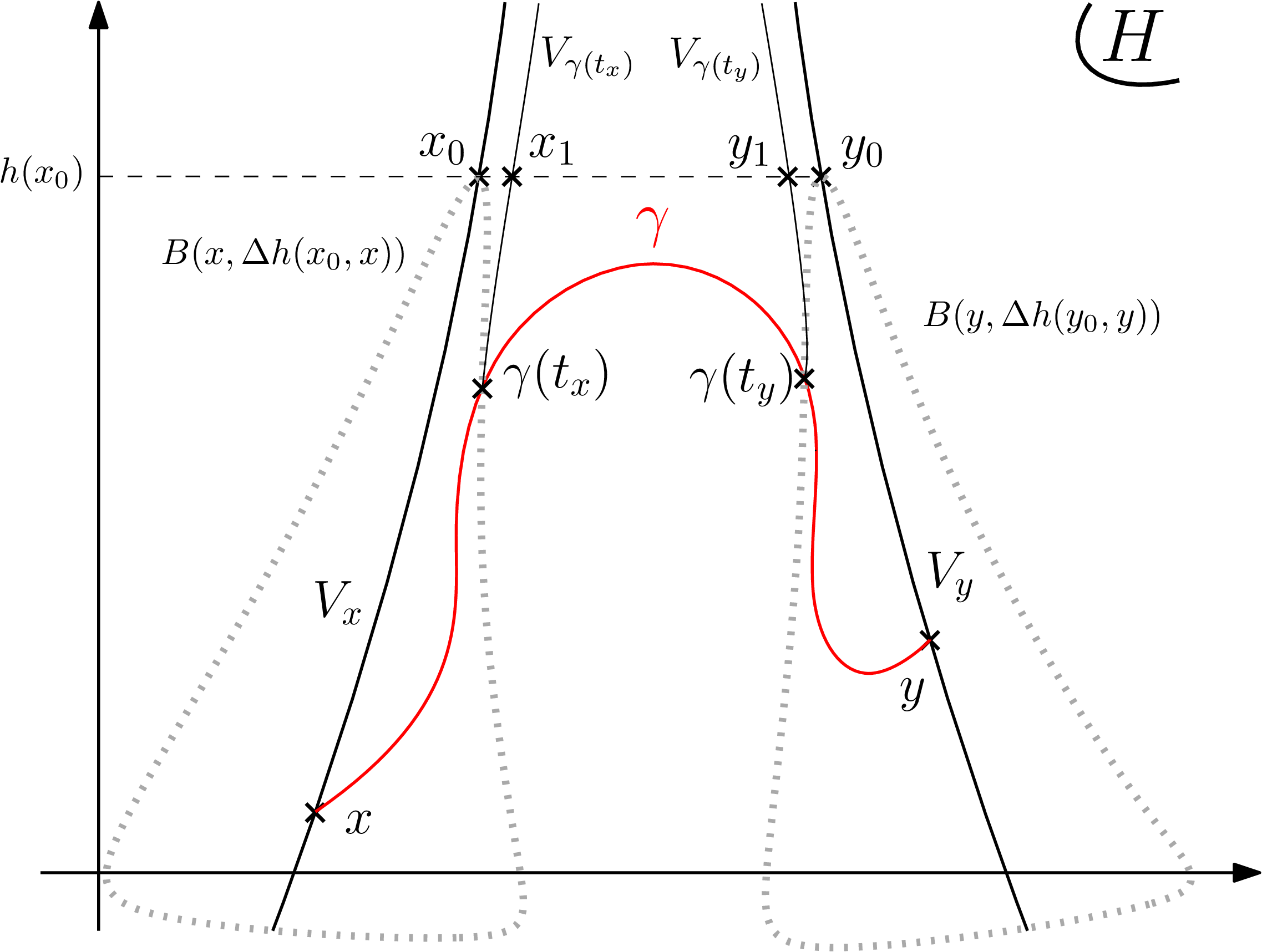} 
\centering
\caption{Proof of Lemma \ref{LemmeAmande}}
\label{FigLemmaAmande}
\end{figure}

For trees (when $\delta=0$) this Lemma still makes sense. Indeed, if $\delta$ tends to $0$ then the length of the path described in this Lemma tends to infinity, which is consistent with the fact that such a path does not exist in trees. The proof would use the fact that in Proposition \ref{LemmeBrid} we have $d(x,\mathrm{im}(c))=0$ when $\delta=0$ since $0$-hyperbolic spaces are real trees. 

\begin{proof}
One can follow the idea of the proof on Figure \ref{FigLemmaAmande}. We will consider $\gamma$ to be parametrised by arclength. Let $B(x,\Delta h(x_0,x))\subset H$ be the ball of radius $h(x_0)-h(x)$  centred on $x$, and let $m\in B(x,\Delta h(x_0,x))$ be a point in this ball. Then:
\begin{equation}
d_r(m,x)=d(m,x)-\Delta h(m,x)\leq \Delta h(x,x_0)-\Delta h(m,x)\leq \Delta h(x_0,m).\nonumber
\end{equation}
Let us first assume that $h(m)\geq h(x)$, then:
\begin{align}
h(m)+\frac{d_r(m,x)}{2}&\leq h(m)+\frac{\Delta h(x_0,m)}{2}\leq h(m)+\frac{h(x_0)-h(m)}{2}=\frac{h(x_0)}{2}+\frac{h(m)}{2} \leq h(x_0).\label{EqDrEtX0}
\end{align}
By Lemma \ref{LEM0} we have:
\begin{align}
d\left(V_x\left(h(m)+\frac{d_r(m,x)}{2}\right),V_m\left(h(m)+\frac{d_r(m,x)}{2}\right)\right)\leq 288\delta.\nonumber
\end{align}
We now assume that $h(m)\leq h(x)$, then:
\begin{equation}
h(x)+\frac{d_r(x,m)}{2}\leq  h(x) +\frac{d(x,m)}{2}\leq h(x)+\frac{\Delta h(x,x_0)}{2}\leq  h(x_0).\nonumber
\end{equation}
Then Lemma \ref{LEM0} provides us with:
\begin{align}
d\left(V_x\left(h(x)+\frac{d_r(m,x)}{2}\right),V_m\left(h(x)+\frac{d_r(m,x)}{2}\right)\right)\leq 288\delta.\nonumber
\end{align}
Since $H$ is a Busemann space, the function $t\to d(V_x(t),V_m(t))$ is convex. Furthermore $t\to d(V_x(t),V_m(t))$ is bounded on $[0;+\infty[$ as $H$ is Gromov hyperbolic, hence $t\to d(V_x(t),V_m(t))$ is a non increasing function. Therefore both cases $h(m)\leq h(x)$ and $h(x)\leq h(m)$ give us that:
\begin{align}
d\Big(x_0,V_m\left(h(x_0)\right)\Big)=d\Big(V_x\left(h(x_0)\right),V_m(h(x_0))\Big)\leq 288\delta.\label{EqDistProjHx0}
\end{align}
In other words, all points of $B(x,\Delta h(x_0,x))$ belong to a vertical geodesic passing nearby $x_0$. By the same reasoning we have $\forall n\in B(y,\Delta h(y_0,y))$ :
\begin{align}
d\Big(y_0,V_n\left(h(y_0)\right)\Big)\leq 288\delta.\label{EqDistProjHx02}
\end{align}
Then by the triangle inequality:
\begin{align}
d\Big(V_m(h(x_0)),V_n (h(y_0))\Big)&\geq -d\Big(x_0,V_m\left(h(x_0)\right)\Big)+d(x_0,y_0)-d\Big(y_0,V_n\left(h(y_0)\right)\Big)\nonumber
\\&\geq 768\delta -288\delta -288\delta\geq 192\delta.\label{EqX1Y155}
\end{align}
Specifically $d(V_m(h(x_0)),V_n (h(y_0)))=d(V_m(h(x_0)),V_n (h(x_0)))>0$ which implies that $m\neq n$. Then $B(x,\Delta h(x_0,x))\cap B(y,\Delta h(y_0,y)) = \emptyset$. By continuity of $\gamma$ we deduce the existence of the two following times $t_x\leq t_y$ such that:
\begin{align}
t_x&=\inf\lbrace t\in[0,T]\ |\ d(\gamma(t),x)=\Delta h(x,x_0)\rbrace,\nonumber
\\t_y&=\sup\lbrace t\in[0,T]\ |\ d(\gamma(t),y)=\Delta h(y,y_0)\rbrace.\nonumber
\end{align}
In order to have a lower bound on the length of $\gamma$ we will need to split this path into three parts: 
\begin{equation}
\gamma=\gamma_{|[0,t_x]}\cup\gamma_{|[t_x,t_y]}\cup\gamma_{|[t_y,T]}.\nonumber
\end{equation}
As $\gamma$ is parametrised by arclength and $d(\gamma(0),\gamma(t_x))=\Delta h(x,x_0)$ we have that:
\begin{equation}
l\left(\gamma_{|[0,t_x]}\right)\geq \Delta h(x,x_0).\label{refXX0}
\end{equation}
For similar reasons we also have:
\begin{equation}
l\left(\gamma_{|[t_y,T]}\right)\geq \Delta h(y,y_0).\label{refYY0}
\end{equation}
We will now focus on proving a lower bound for the length of $\gamma_{|[t_x,t_y]}$.
\\\\We want to construct a path $\gamma'$ joining $x_1=V_{\gamma(t_x)}(h(x_0))$ to $y_1=V_{\gamma(t_y)}(h(x_0))$, that stays below $h(x_0)$ and such that $\gamma_{|[t_x,t_y]}$ is contained in $\gamma'$. Let $x_1:=V_{\gamma(t_x)}(h(x_0))$ and $y_1:=V_{\gamma(t_y)}(h(x_0))$. We construct $\gamma'$ by gluing paths together:
\begin{equation}
\gamma' = 
\left\{
    \begin{array}{ll}
        V_{\gamma(t_x)} & \mbox{from } x_1 \mbox{ to } \gamma(t_x)\\
        \gamma & \mbox{from } \gamma(t_x) \mbox{ to } \gamma(t_y)\\
        V_{\gamma(t_y)} & \mbox{from } \gamma(t_y) \mbox{ to } y_1\\
    \end{array}
\right.\nonumber
\end{equation}
Applying inequalities (\ref{EqDistProjHx0}) and (\ref{EqDistProjHx02}) used on $\gamma(t_x)$ and $\gamma(t_y)$ we get:
\begin{align}
d(x_0,x_1)&\leq 288\delta ,\label{EqX01108}
\\d(y_0,y_1)&\leq 288\delta.\label{EqY01108}
\end{align}

In order to apply Proposition \ref{LemmeBrid} to $\gamma'$ we need to check that there exists a point $A$ of the geodesic segment $[x_1,y_1]$ such that $h(A)\geq h(x_0)$. Applying Lemma \ref{LEM0} to $[x_1,y_1]$ and since $h(x_1) = h(y_1)$ we get: 
\begin{align}
h^+([x_1,y_1])&\geq \frac{d_r(x_1,y_1)}{2}+h(x_0)-96\delta= \frac{d(x_1,y_1)}{2}+h(x_0)-96\delta.\nonumber
\end{align}
Thanks to the triangle inequality and inequalities (\ref{EqX01108}) and (\ref{EqY01108}):
\begin{align}
h^+([x_1,y_1])&\geq \frac{d(y_0,x_0)-d(x_0,x_1)-d(y_0,y_1)}{2}+h(x_0)-96\delta\geq \frac{d(x_0,y_0)}{2}+h(x_0)-384\delta.\nonumber
\end{align} 
Since by hypothesis $d(x_0,y_0)>768\delta$, there exists a point $A$ of $[x_1,y_1]$ exactly at the height:
\begin{equation}
h(A)=\frac{d(x_0,y_0)}{2}+h(x_0)-384\delta.\nonumber
\end{equation}
We can then apply Proposition \ref{LemmeBrid} to get:
\begin{align*}
\delta | \log_2(l(\gamma'))|+1&\geq d(A,\gamma')\geq\Delta h(A,x_0)\geq\frac{d(x_0,y_0)}{2}+h(x_0)-384\delta-h(x_0)
\\&\geq\frac{d(x_0,y_0)}{2}-384\delta. 
\end{align*}
Since $\delta\geq 1$, last inequality implies that $l(\gamma')\geq 2^{-385}2^{\frac{1}{2\delta}d(x_0,y_0)}$. Now we use this inequality to have a lower bound on the length of $\gamma_{|[t_x,t_y]}$:
\begin{align}
l(\gamma_{|[t_x,t_y]})&\geq l(\gamma')-\Delta h(\gamma(t_x),x_0)-\Delta h(\gamma(t_y),y_0)\nonumber
\\&\geq 2^{-385}2^{\frac{1}{2\delta}d(x_0,y_0)}-\Delta h(\gamma(t_x),x_0)-\Delta h(\gamma(t_y),y_0).\label{FirstPieceAdding}
\end{align}
We claim that $l\left(\gamma_{|[t_x,t_y]}\right)\geq  \Delta h(\gamma(t_x),x_0)+\Delta h(\gamma(t_y),y_0)-48\delta$, hence:
\begin{align}
l\left(\gamma_{|[t_x,t_y]}\right)&\geq 2^{-386}2^{\frac{1}{2\delta}d(x_0,y_0)}-24\delta,\label{IneqUseHereDontNow}
\end{align}
which ends the proof by combining inequality (\ref{IneqUseHereDontNow}) with inequalities (\ref{refXX0}) and (\ref{refYY0}).
\\\\\hspace*{0.6cm}Proof of the claim. Inequality (\ref{EqX1Y155}) with $m=\gamma(t_x)$ and $n=\gamma(t_y)$ gives $d(x_1,y_1)\geq 192\delta$. We want to prove that $h^+([\gamma(t_x),\gamma(t_y)])\geq h(x_1)-24\delta$. First, by Lemma \ref{ThinTrianglePropr} we have that $[\gamma(t_x),\gamma(t_y)]\cup V_{\gamma(t_x)} \cup V_{\gamma(t_y)}$ is a $24\delta$-slim triangle. Then there exist three times $t_0$, $t_1$ and $t_2$ such that $d\left(V_{\gamma(t_x)}(t_1),\gamma (t_0)\right)\leq 24\delta$ and such that $d\left(V_{\gamma(t_y)}(t_2),\gamma (t_0)\right)\leq 24\delta$. Then:
\begin{align}
|t_1-t_2|&=\Delta h\left( V_{\gamma(t_x)}(t_1),V_{\gamma(t_y)}(t_2)\right)\leq d\left(V_{\gamma(t_x)}(t_1),V_{\gamma(t_y)}(t_2)\right)\nonumber
\\&\leq d\left(V_{\gamma(t_x)}(t_1),\gamma (t_0)\right)+d\left(\gamma (t_0),V_{\gamma(t_y)}(t_2)\right)\leq 48\delta.\label{InegT1T218D}
\end{align}
We will show by contradiction that either $t_1=h(V_{\gamma(t_x)}(t_1))\geq h(x_0)$ or $t_2=h(V_{\gamma(t_y)}(t_2))\geq h(x_0)$.
\\Assume that $t_1< h(x_0)$ and $t_2< h(x_0)$. Then by the triangle inequality:
\begin{align*}
d\big(V_{\gamma(t_x)}(t_1),V_{\gamma(t_y)}(t_2)\big)&\geq d\big(V_{\gamma(t_y)}(t_2),V_{\gamma(t_x)}(t_2)\big)-d\big(V_{\gamma(t_x)}(t_2),V_{\gamma(t_x)}(t_1)\big)
\\&\geq d\big(V_{\gamma(t_y)}(t_2),V_{\gamma(t_x)}(t_2)\big)-48\delta\text{,  since }|t_1-t_2|\leq 48\delta\text{ by equation (\ref{InegT1T218D}).}
\end{align*}
As $H$ is a Busemann space, the function $t\mapsto d\big(V_{\gamma(t_x)}(t),V_{\gamma(t_y)}(t)\big)$ is non increasing (convex and bounded function). Furthermore, $h(x_0)\geq t_2$ hence:
\begin{align*}
48\delta&\geq d\big(V_{\gamma(t_x)}(t_1),V_{\gamma(t_x)}(t_2)\big)\geq  d\big(V_{\gamma(t_x)}(t_2),V_{\gamma(t_y)}(t_2)\big)-48\delta
\\&\geq d\big(V_{\gamma(t_x)}(h(x_0)),V_{\gamma(t_y)}(h(x_0))\big)-48\delta \geq d(x_1,y_1)-48\delta
\\&\geq d(x_0,y_0)-d(x_0,x_1)-d(y_0,y_1)-48\delta\geq d(x_0,y_0)-624\delta\text{, by inequalities }(\ref{EqX01108})\text{ and }(\ref{EqY01108}),
\\&\geq 49\delta\text{, since }d(x_0,y_0)\geq 768\delta\text{ by assumption},
\end{align*}
which is impossible. Therefore $t_1\geq h(x_0)$ or $t_2\geq h(x_0)$. We assume without loss of generality that $t_1\geq h(x_0)$, then:
\begin{equation}
\Delta h\big(\gamma(t_0),V_{\gamma(t_x)}(t_1)\big) \leq d\big(\gamma(t_0),V_{\gamma(t_x)}(t_1)\big)\leq 24\delta,\nonumber
\end{equation}
which implies:
\begin{align*}
h^+([\gamma(t_x),\gamma(t_y)])\geq h(\gamma(t_0))\geq h\left(V_{\gamma(t_x) }(t_1)\right)-\Delta h\big(\gamma(t_0),V_{\gamma(t_x)}(t_1)\big)\geq h(x_0)-24\delta,
\end{align*}
and gives us:
\begin{align}
l\left(\gamma_{|[t_x,t_y]}\right)&\geq h^+([\gamma(t_x),\gamma(t_y)])-h(\gamma(t_x))+h^+([\gamma(t_x),\gamma(t_y)])-h(\gamma(t_y))\nonumber
\\&\geq h(x_0)-24\delta-h(\gamma(t_x))+h(x_0)-24\delta-h(\gamma(t_y))\nonumber
\\&\geq \Delta h(\gamma(t_x),x_0)+\Delta h(\gamma(t_y),y_0)-48\delta.\label{SecondPieceAdding}
\end{align}
\end{proof}

%CONTROLE DE LA DISTANCE RELATIVE GOING BACKWARD
Next lemma shows that we are able to control the relative distance of a couple of points travelling along two vertical geodesics. We recall that for all $a,b\in H$, $d_r(a,b)=d(a,b)-\Delta h(a,b)$.

\begin{lemma}[Backwards control]\label{ControleRelativeDistGoingBackward}
Let $\delta\geq 0$ and $H$ be a proper, $\delta$-hyperbolic, Busemann space. Let $V_1$ and $V_2$ be two vertical geodesics of $H$. Then for all couple of times $(t_1,t_2)$ and for all $t\in\left[0,\frac{1}{2}d_r(V_1(t_1),V_2(t_2))\right]$: 
\begin{align*}
\left|d_r\left(V_1\left(t_1+\frac{1}{2}d_r(V_1(t_1),V_2(t_2))-t\right),V_2\left(t_2+\frac{1}{2}d_r(V_1(t_1),V_2(t_2))-t\right)\right)-2t\right|\leq 288\delta.
\end{align*}
\end{lemma}

\begin{figure}
\begin{center}
\includegraphics[scale=0.35]{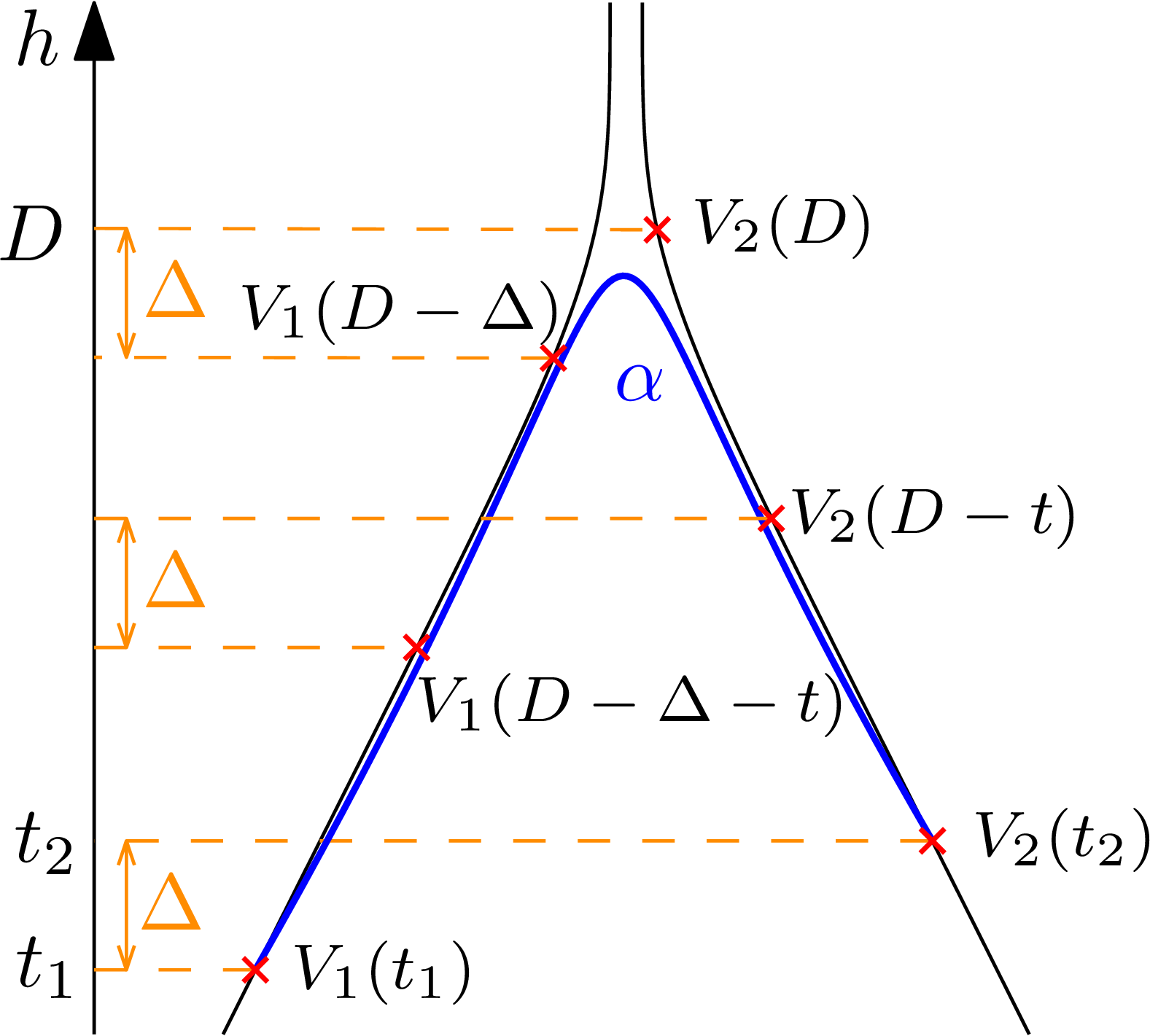} 
\caption{Proof of Lemma \ref{ControleRelativeDistGoingBackward}}\label{FigBackwardControl}
\end{center}
\end{figure}

\begin{proof}
To simplify the computations, we use the following notations, $D:=t_2+\frac{1}{2}d_r(V_1(t_1),V_2(t_2))$ and $\Delta=|t_1-t_2| $. The term $\Delta$ is the difference of height between $V_1(t_1)$ and $V_2(t_2)$ since vertical geodesics are parametrised by their height. Then we have to prove that $\forall t\in\left[0,\frac{1}{2}d_r(V_1(t_1),V_2(t_2))\right]$, $|d_r(V_1(D-\Delta-t),V_2(D-t))-2t|\leq 288\delta$. We can assume without loss of generality that $t_1\leq t_2$. Lemma \ref{LEM0} applied with $x=V_1(t_1)$ and with $y=V_2(t_2)$ gives us $d(V_1(D),V_2(D))\leq 288\delta$. Furthermore, the relative distance is smaller than the distance, hence $d_r(V_1(D),V_2(D))\leq 288\delta$. Now, if we move the two points backward from $V_1(D-\Delta)$ and $V_2(D)$ along $V_1$ and $V_2$, we have for $t\in[0,D]$:
\begin{align}
d_r(V_1(D-\Delta-t),V_2(D-t))=&d(V_1(D-\Delta-t ),V_2(D-t))-\Delta
\\\leq &d(V_1(D-\Delta-t ),V_1(D-\Delta))+d(V_1(D-\Delta),V_2(D))\nonumber
\\&+d(V_2(D ),V_2(D-t))-\Delta, \nonumber
\\&\text{furthermore }V_1\text{ and }V_2\text{ are geodesics, then:}\nonumber
\\\leq &t+d(V_1(D-\Delta),V_1(D))+d(V_1(D),V_2(D))+t-\Delta  \nonumber
\\\leq &t+\Delta+288\delta+t-\Delta\leq 2t+288\delta.\label{ControlDrelatDown}
\end{align}
Let us consider a geodesic $\alpha$ between $V_1(t_1)$ and $V_2(t_2)$. Since $H$ is a Busemann space, and thanks to Lemma \ref{LEM0} we have $d\left(V_1(D-\Delta-t),\alpha(D-\Delta-t_1-t)\right)\leq 144\delta$ and $d\left(V_2(D-t),\alpha(D-t_1+t)\right)\leq 144\delta$. Then the second part of our inequality follows:
\begin{align}
d_r(V_1(D-\Delta-t),V_2(D-t))=&d(V_1(D-\Delta-t ),V_2(D-t))-\Delta \nonumber
\\\geq& d(\alpha(D-\Delta-t_1-t),\alpha(D-t_1+t))\nonumber
\\&-d(V_1(D-\Delta-t ),\alpha(D-\Delta-t_1-t))\nonumber
\\&-d(V_2(D-t),\alpha(D-t_1+t))-\Delta\nonumber
\\\geq& d(\alpha(D-\Delta-t_1-t),\alpha(D-t_1+t)) -288\delta-\Delta\nonumber
\\\geq& 2t+\Delta -288\delta-\Delta\geq 2t -288\delta.\label{ControlDrelatUp}
\end{align}
\end{proof}

The next lemma is a slight generalisation of Lemma \ref{LemmeAmande}. The difference being that we control the length of a path with its maximal height instead of the distance between the projection of its extremities on a horosphere.

\begin{lemma}\label{ExpLengthWhenBelowSameHeight}
Let $\delta\geq 1$ and $H$ be a proper, $\delta$-hyperbolic, Busemann space. Let $x,y\in H$ such that $h(x)\leq h(y)$. Let $\alpha$ be a path connecting $x$ to $y$ with $h^+(\alpha)\leq h(y)+\frac{1}{2}d_r(x,y)-\Delta H$ and where $\Delta H$ is a positive number such that $\Delta H> 555\delta$. Then:
\begin{align*}
l(\alpha)\geq d(x,y)+2^{-530}2^{\frac{1}{\delta}\Delta H}-2\Delta H-24\delta.
\end{align*} 
\end{lemma}

\begin{figure}
\begin{center}
\includegraphics[scale=0.35]{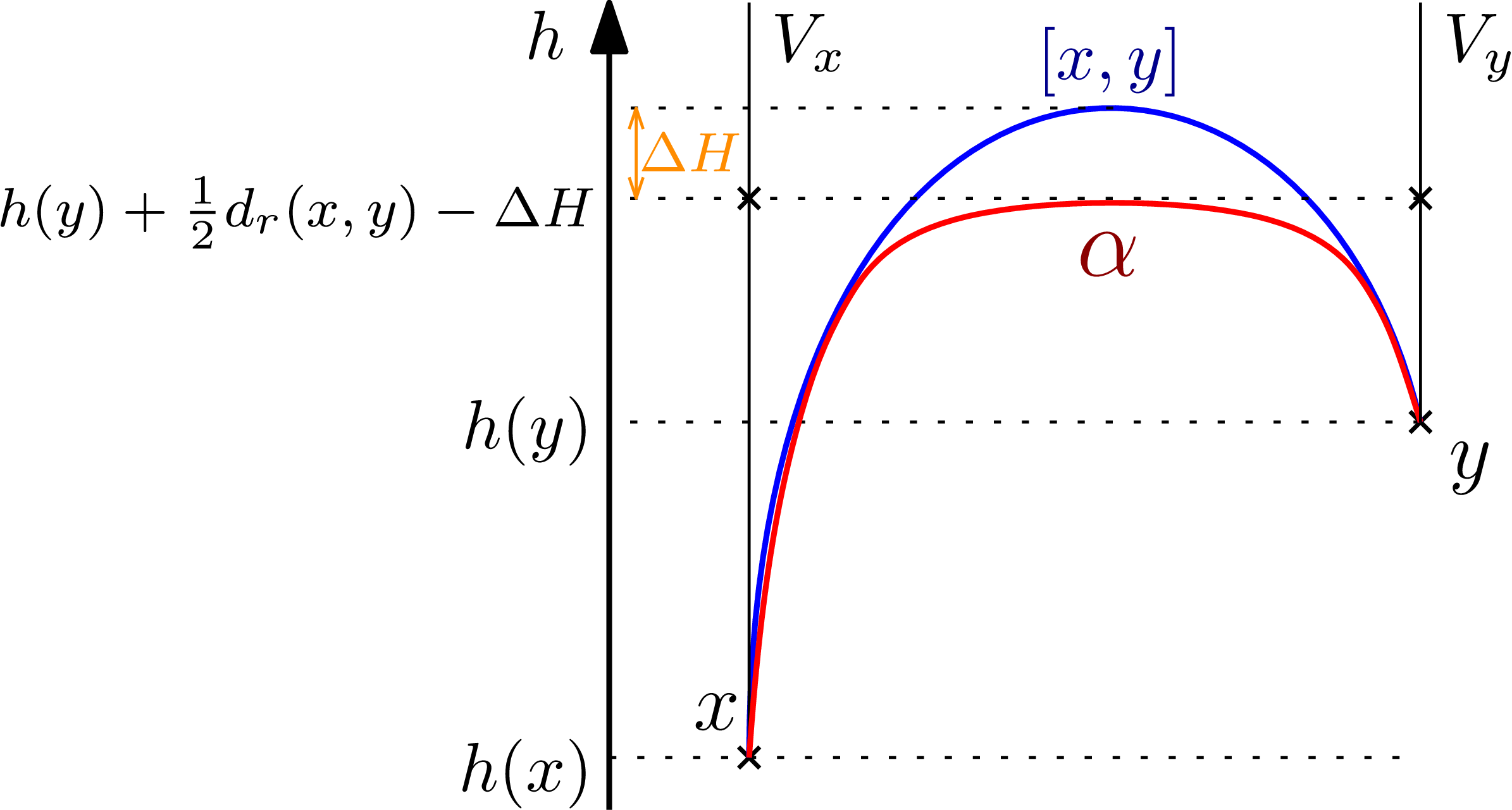}
\captionof{figure}{Proof of Lemma \ref{ExpLengthWhenBelowSameHeight}}
\label{FigExpLengthWhenBelowSameHeight}
\end{center}
\end{figure}

\begin{proof}
This proof is illustrated in Figure \ref{FigExpLengthWhenBelowSameHeight}. Since $h^+(\alpha)\geq h(y)$ we have that $\frac{1}{2}d_r(x,y)\geq \Delta H$. Applying Lemma \ref{ControleRelativeDistGoingBackward} with $V_1=V_x$, $V_2=V_y$, $t_1=h(x)$, $t_2=h(y)$ and $t=\Delta H$ we have:
\begin{align*}
\left|d_r\left(V_x\left(h(x)+\frac{1}{2}d_r(x,y)-\Delta H\right),V_y\left(h(y)+\frac{1}{2}d_r(x,y)-\Delta H\right)\right)-2\Delta H\right|\leq 288\delta.
\end{align*}
Then we have:
\begin{align*}
d_r\left(V_x\left(h(x)+\frac{1}{2}d_r(x,y)-\Delta H\right),V_y\left(h(y)+\frac{1}{2}d_r(x,y)-\Delta H\right)\right)\geq 2\Delta H- 288\delta.
\end{align*}
Furthermore, Lemma \ref{LinkDrAndDSameHeight} applied on $V_x\left(h(x)+\frac{1}{2}d_r(x,y)-\Delta H\right)$ and $V_y\left(h(y)+\frac{1}{2}d_r(x,y)-\Delta H\right)$ gives (notice that the only difference between the two sides of the following inequality is the height in the vertical geodesic $V_x$):
\begin{align*}
&d_r\left(V_x\left(h(x)+\frac{1}{2}d_r(x,y)-\Delta H\right),V_y\left(h(y)+\frac{1}{2}d_r(x,y)-\Delta H\right)\right)\\&\leq d\left(V_x\left(h(y)+\frac{1}{2}d_r(x,y)-\Delta H\right),V_y\left(h(y)+\frac{1}{2}d_r(x,y)-\Delta H\right)\right)+54\delta.
\end{align*} 
Then:
\begin{align}
d\left(V_x\left(h(y)+\frac{1}{2}d_r(x,y)-\Delta H\right),V_y\left(h(y)+\frac{1}{2}d_r(x,y)-\Delta H\right)\right)\geq 2\Delta H -342\delta>768\delta.\label{ControleDHInproof}
\end{align}
Let us denote $t_0=h(y)+\frac{1}{2}d_r(x,y)-\Delta H$. Thanks to inequality (\ref{ControleDHInproof}) the hypothesis of Lemma \ref{LemmeAmande} holds with $x_0=V_x\left(h(y)+\frac{1}{2}d_r(x,y)-\Delta H\right)$ and $y_0=V_y\left(h(y)+\frac{1}{2}d_r(x,y)-\Delta H\right) $. Applying this lemma on $\alpha$ provides:
\begin{align*}
l(\alpha)&\geq\Delta h(x,x_0)+\Delta h(y,y_0) +2^{-386}2^{\frac{1}{2\delta}d(x_0,y_0)}-24\delta
\\&\geq h(y)+\frac{1}{2}d_r(x,y)-\Delta H-h(x)+h(y)+\frac{1}{2}d_r(x,y)-\Delta H-h(y)+2^{-386}2^{\frac{1}{2\delta}d(x_0,y_0)}-24\delta 
\\&\geq \Delta h(y,x)+d_r(y,x)-2\Delta H+2^{-386}2^{\frac{1}{2\delta}d(x_0,y_0)}-24\delta 
\\&\geq d(x,y)-2\Delta H+2^{-386}2^{\frac{1}{2\delta}(2\Delta H -288\delta)}-24\delta \text{, by equation (\ref{ControleDHInproof}).}
\\&\geq d(x,y)+2^{-530}2^{\frac{1}{\delta}\Delta H}-2\Delta H-24\delta .
\end{align*}
\end{proof}

This previous lemma tells us that a path needs to reach a sufficient height for its length not to increase to much. We give now a generalisation of Lemma \ref{ExpLengthWhenBelowSameHeight}, where the path reaches a given low height before going to its end point. This proposition will be the central result for the understanding of the geodesic shapes in a horospherical product.

\begin{propo}\label{ExpLengthWhenBelowAndReachPoint}
Let $\delta\geq 1$ and $H$ be a proper, $\delta$-hyperbolic, Busemann space. Let $x,y,m\ \in H$ such that $h(m)\leq h(x)\leq h(y)$ and let $\alpha:[0,T]\to H$ be a path connecting $x$ to $y$ such that $h^-(\alpha)=h(m)$. With the notation $\Delta H=h(y)+\frac{1}{2}d_r(x,y)-h^+(\alpha)$ we have:
\begin{align*}
l(\alpha)\geq 2\Delta h(x,m)+d(x,y)+2^{-850}2^{\frac{1}{\delta}\Delta H}-1-\max(0,2\Delta H)-1700\delta.
\end{align*} 
\end{propo}

\begin{figure}
\begin{center}
\includegraphics[scale=0.35]{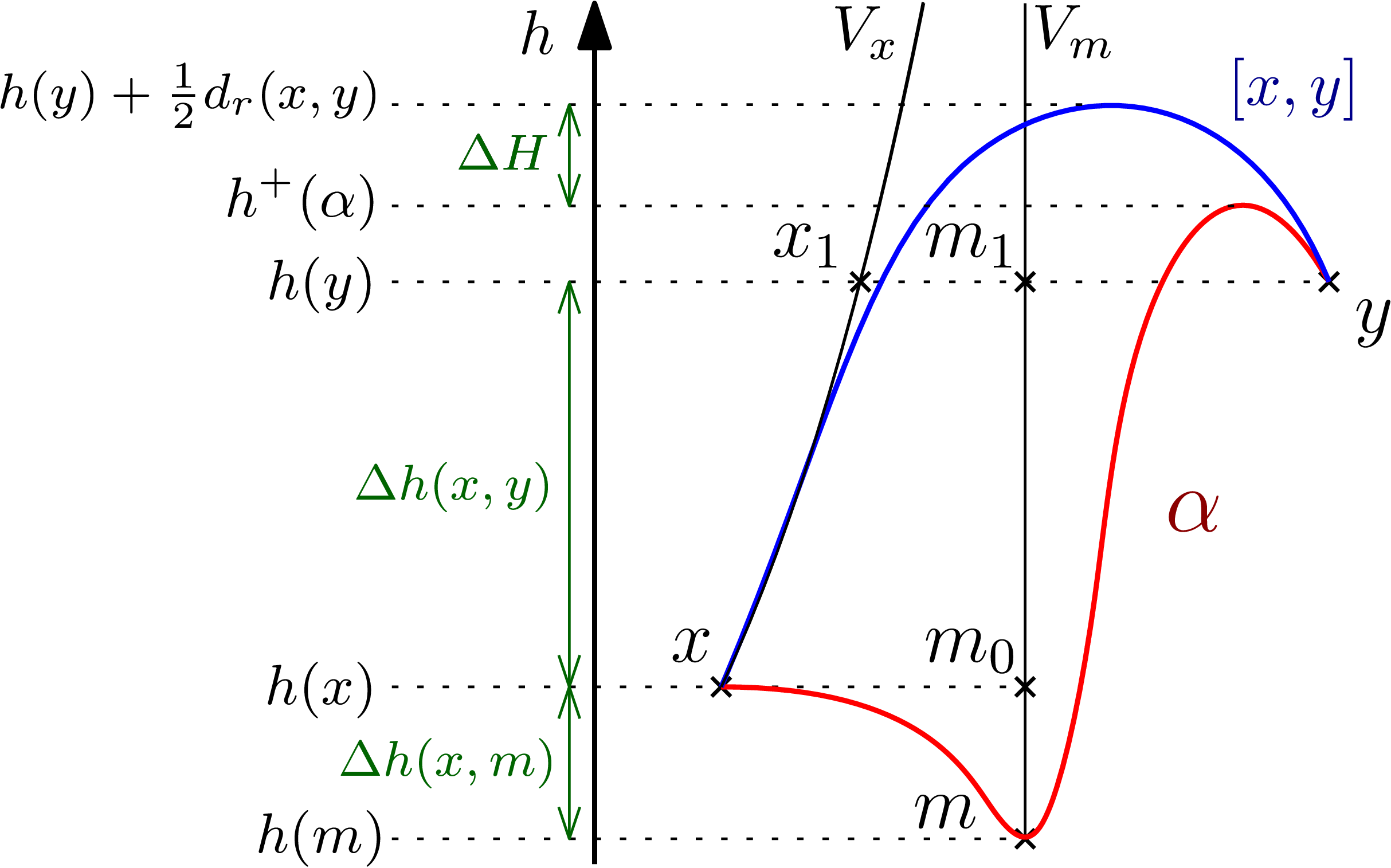}
\captionof{figure}{Proof of Proposition \ref{ExpLengthWhenBelowAndReachPoint}}
\label{FigExpLengthWhenBelowAndReachPoint}
\end{center}
\end{figure}

\begin{proof}
This proof is illustrated in Figure \ref{FigExpLengthWhenBelowAndReachPoint}. We first assume that $\Delta H>850\delta$, we postpone the other cases to the end of this proof. Let $V_x$ and $V_m$ be vertical geodesics respectively containing $x$ and $m$. We call $x_1=V_x(h(y))$ and $m_1=V_m(h(y))$ the points of $V_x$ and $V_m$ at height $h(y)$. First, Lemma \ref{LinkDrAndDSameHeight} provides $|d(x_1,y)-d_r(x,y)|\leq 54\delta$. Then we consider a geodesic triangle between the three points $x_1$, $m_1$ and  $y$. Lemma \ref{LEM0} tells us that $h^+([x_1,y])\geq h(y)+\frac{1}{2} d_r(x_1,y)-96\delta\geq h(y)+\frac{1}{2}d_r(x,y)-123\delta$. Since $[x_1,y]$ is included in the $\delta$-neighbourhood of the two other sides of the geodesic triangle, one of the two following inequalities holds:
\begin{align*}
1)&\ h^+([x_1,m_1])\geq h(y)+\frac{1}{2}d_r(x,y)-124\delta
\\2)&\ h^+([m_1,y])\geq h(y)+\frac{1}{2}d_r(x,y)-124\delta.
\end{align*} 
\hspace*{0.6cm}We first assume $1)$ that $h^+([x_1,m_1])\geq h(y)+\frac{1}{2}d_r(x,y)-124\delta$, hence:
\begin{equation}
d(x_1,m_1)\geq d_r(x,y)-248\delta.\label{UseInRP2}
\end{equation}
Let us denote $m_0=V_m(h(x))$ the point of $V_m$ at height $h(x)$. By considering the $2\delta $-slim quadrilateral between the points $x,x_1,m_0,m_1$ we have that $[x_1,m_1]$ is in the $2\delta$- neighbourhood of $[x_1,x]\cup[x,m_0]\cup[m_0,m]$. Furthermore $d_r(x,y)\geq2(h^+(\alpha)-h(y))+2\Delta H\geq 2\Delta H\geq 1700\delta$ by assumption, then $h^+([x_1,m_1])\geq h(y)+\frac{1}{2}d_r(x,y)-124\delta\geq h(y)+726\delta$. Since $h^+([x_1,x])=h^+([m_0,m_1])=h(y)$ we have that $h^+([x,m_0])\geq h^+([x_1,m_1])-2\delta\geq h(y)+724\delta$. Moreover:
\begin{align*}
d_r(x,m_0)=d(x,m_0)\geq h^+([x,m_0])-h(x)\geq h(y)-h(x)+724\delta \geq \Delta h(x,y)+724\delta,
\end{align*} 
which allows us to use Lemma \ref{ControleRelativeDistGoingBackward} on $V_x$ and $V_m$ with $t=\frac{1}{2}d_r(x,m_0)-\Delta h(x,y)\geq 0$ and $t_1=t_2=h(x)$. It gives:
\begin{align*}
\left| d_r\Big(V_x\big(h(x)+\Delta h(x,y)\big),V_m\big(h(x)+\Delta h(x,y)\big)\Big)-d_r(x,m_0)+2\Delta h(x,y)\right|\leq 288\delta,
\end{align*}
which implies in particular:
\begin{equation}
d_r\Big(V_x\big(h(y)\big),V_m\big(h(y)\big)\Big)+2\Delta h(x,y)-288\delta\leq d_r(x,m_0).\label{UseInRP3}
\end{equation} 
Combining inequalities (\ref{UseInRP2}) and (\ref{UseInRP3}) we have $d(x,m_0)=d_r(x,m_0)\geq d_r(x,y)+2\Delta h(x,y)-536\delta$. Lemma \ref{LinkDrAndDSameHeight} used on $x$ and $m$ then gives:
\begin{align}\label{UseInRP1}
d_r(x,m)\geq d(x,m_0)-54\delta\geq d_r(x,y)+2\Delta h(x,y)-590\delta.
\end{align} 
Let us denote $\alpha_1$ the part of $\alpha$ linking $x$ to $m$ and $\alpha_2$ the part of $\alpha$ linking $m$ to $y$. We have:
\begin{align*}
h^+(\alpha_1)\leq& h^+(\alpha)\leq h(y)+\frac{1}{2}d_r(x,y)-\Delta H \leq h(x)+\Delta h(x,y)+\frac{1}{2}d_r(x,y)-\Delta H 
\\\leq& h(x) +\frac{1}{2}\left(2\Delta h(x,y)+d_r(x,y)\right)-\Delta H \leq h(x)+\frac{1}{2}\left(d_r(x,m)+590\delta\right)-\Delta H \text{, by inequality }(\ref{UseInRP1}).
\\\leq& h(x)+\frac{1}{2}d_r(x,m)+295\delta-\Delta H \leq h(x)+\frac{1}{2}d_r(x,m)-\Delta H',
\end{align*}
with $\Delta H '=\Delta H-295\delta$. By assumption $\Delta H>850\delta$, hence $\Delta H'> 555 \delta$ which allows us to apply Lemma \ref{ExpLengthWhenBelowSameHeight} on $\alpha_1$. It follows:
\begin{align*}
l(\alpha_1)\geq&  d(x,m)+2^{-530}2^{\frac{1}{\delta}\Delta H'}-2\Delta H'-24\delta 
\\\geq & \Delta h(x,m)+d_r(x,m)+2^{-825}2^{\frac{1}{\delta}\Delta H}-2\Delta H-614\delta \text{, since }\Delta H'=\Delta H-295\delta.
\\\geq & \Delta h(x,m)+d_r(x,y)-590\delta+2^{-825}2^{\frac{1}{\delta}\Delta H}-2\Delta H-614\delta \text{, by inequality (\ref{UseInRP1})}
\\\geq & \Delta h(x,m)+d_r(x,y)+2^{-825}2^{\frac{1}{\delta}\Delta H}-2\Delta H-1204\delta .
\end{align*} 
We use in the following inequalities that $l(\alpha_2)\geq d(m,y)\geq\Delta h(m,y)$, we have:
\begin{align*}
l(\alpha)&\geq l(\alpha_1)+l(\alpha_2) \geq \Delta h(x,m)+d_r(x,y)+2^{-825}2^{\frac{1}{\delta}\Delta H}-2\Delta H-1204\delta  +\Delta h (m,y)
\\&\geq 2\Delta h(x,m)+\Delta h(x,y)+d_r(x,y)+2^{-825}2^{\frac{1}{\delta}\Delta H}-2\Delta H-1204\delta
\\&\geq 2\Delta h(x,m)+d(x,y)+2^{-825}2^{\frac{1}{\delta}\Delta H}-2\Delta H-1204\delta
\\&\geq 2\Delta h(x,m)+d(x,y)+2^{-850}2^{\frac{1}{\delta}\Delta H}-1-2\Delta H-1700\delta,
\\&\geq 2\Delta h(x,m)+d(x,y)+2^{-850}2^{\frac{1}{\delta}\Delta H}-1-\max(0,2\Delta H)-1700\delta\text{, since }\Delta H>850\delta\geq 0,
\end{align*}
which ends the proof for case 1).
\\\\Now assume that $2)$ holds, which is $h^+([m_1,y])\geq h(y)+\frac{1}{2}d_r(x,y)-124\delta$. It implies $d(m_1,y)\geq d_r(x,y)-248\delta$, then:
\begin{align*}
h^+(\alpha_2)\leq& h^+(\alpha)\leq h(y)+\frac{1}{2}d_r(x,y)-\Delta H \leq h(y)+\frac{1}{2}d_r(m_1,y)+124\delta-\Delta H 
\\&\leq h(y)+\frac{1}{2}d_r(m_1,y)-\Delta H '',
\end{align*}
with $\Delta H ''=\Delta H-124\delta$. Lemma \ref{LinkDrAndDSameHeight} provides us with:
\begin{equation}
d_r(m,y)\geq d(m_1,y)-54\delta\geq d_r(x,y)-302\delta.\label{UseInRP4}
\end{equation}
Since $\Delta H> 850\delta$, we have $\Delta H''> 726\delta$ which allows us to apply Lemma \ref{ExpLengthWhenBelowSameHeight} on $\alpha_2$. It follows that:
\begin{align*}
l(\alpha_2)\geq&  d(y,m)+2^{-530}2^{\frac{1}{\delta}\Delta H''}-2\Delta H''-24\delta 
\\\geq & \Delta h(y,m)+d_r(y,m)+2^{-654}2^{\frac{1}{\delta}\Delta H}-2\Delta H-272\delta \text{, since }\Delta H''=\Delta H-124\delta.
\\\geq & \Delta h(y,m)+d_r(x,y)+2^{-654}2^{\frac{1}{\delta}\Delta H}-2\Delta H-574\delta\text{, by inequality (\ref{UseInRP3})} .
\end{align*} 
Hence:
\begin{align*}
l(\alpha)&\geq l(\alpha_1)+l(\alpha_2) \geq \Delta h(x,m)+\Delta h(y,m)+d_r(x,y)+2^{-654}2^{\frac{1}{\delta}\Delta H}-2\Delta H-574\delta 
\\&\geq 2\Delta h(x,m)+\Delta h(y,x)+d_r(x,y)+2^{-654}2^{\frac{1}{\delta}\Delta H}-2\Delta H-574\delta 
\\&\geq 2\Delta h(x,m)+d(x,y)+2^{-654}2^{\frac{1}{\delta}\Delta H}-2\Delta H-574\delta 
\\&\geq 2\Delta h(x,m)+d(x,y)+2^{-850}2^{\frac{1}{\delta}\Delta H}-1-\max(0,2\Delta H)-1700\delta.
\end{align*}
There remains to treat the case when $\Delta H \leq 850\delta$, where $\Delta H=h(y)+\frac{1}{2}d_r(x,y)-h^+(\alpha)$. Let $n$ denote a point of $\alpha$ such that $h(n)=h^+(\alpha)$. If $m$ comes before $n$, we have $l(\alpha)\geq d(x,m)+d(m,n)+d(n,y)$. Otherwise $n$ comes before $m$ and we have $l(\alpha)\geq d(x,n)+d(n,m)+d(m,y)$. Since $h(m)\leq h(x)\leq h(y)\leq h(n)$ we always have:
\begin{align*}
l(\alpha)&\geq \Delta h(x,m)+\Delta h(m,n)+\Delta h(n,y)
\\&\geq \Delta h(x,m) + \Delta h(m,x) +\Delta h(x,y)+\Delta h(y,n) + \Delta h(y,n)
\\&\geq 2\Delta h(x,m)  +\Delta h(x,y)+2(h^+(\alpha)-h(y))
\\&\geq 2\Delta h(x,m)  +\Delta h(x,y)+d_r(x,y)-2\Delta H\geq 2 \Delta h(m,x) +d(x,y)-1700\delta.
\end{align*}
Furthermore $\Delta H \leq 850\delta$, then $2^{-850}2^{\frac{1}{\delta}\Delta H}\leq 1$. Therefore:
\begin{align*}
l(\alpha)&\geq 2 \Delta h(m,x) +d(x,y)+2^{-850}2^{\frac{1}{\delta}\Delta H}-1-\max(0,2\Delta H)-1700\delta,
\end{align*}
which ends the proof for the remaining case.
\end{proof}

\subsection{Length of geodesic segments in horospherical products}

From now on, unless otherwise specified, $X$ and $Y$ will always be two proper, geodesically complete, $\delta$-hyperbolic, Busemann spaces with $\delta\geq 1$, and $N$ will always be an admissible norm. Let $p$ and $q$ be two points of $X\bowtie Y$, and let $\alpha$ be a geodesic of $ X\bowtie Y $ connecting them. We first prove an upper bound on the length of $\alpha$ by computing the length of a path $\gamma\subset X\bowtie Y $ linking $p$ to $q$ 

\begin{lemma}\label{PathConnectingPointsInHoroProduct}
Let $p=(p_X,p_Y)$ and $q=(q_X,q_Y)$ be points of the horospherical product $X\bowtie  Y$. There exists a path $\gamma$ connecting $p$ to $q$ such that:
\begin{align*}
l_N(\gamma)\leq d_r(p_Y,q_Y)+d_r(p_X,q_X)+\Delta h(p,q)+1152\delta C_N.
\end{align*}
\end{lemma}

\begin{proof}

Without loss of generality, we assume $h(p)\leq h(q)$. One can follow the idea of the proof on Figure \ref{FigureConstructionPathHoroProduct}. We consider $V_{p_X}$ and $V_{q_X}$ two vertical geodesics of $X$ containing $p_X$ and $q_X$ respectively. Similarly let $V_{p_Y}$ and $V_{q_Y}$ be two vertical geodesics of $Y$ containing $p_Y$ and $q_Y$ respectively. We will use them to construct $\gamma$. Let $A_1$ be the point of the vertical geodesic $(V_{p_X},V_{p_Y})\subset X\bowtie Y $ at height $h(p)-\frac{1}{2}d_r(p_Y,q_Y)$ and $A_2$ be the point of the vertical geodesic $(V_{p_X},V_{q_Y})\subset X\bowtie Y $ at the same height $h(p)-\frac{1}{2}d_r(p_Y,q_Y)$. Let $A_3$ be the point of the vertical geodesic $(V_{p_X},V_{q_Y})$ at height $h(q)+\frac{1}{2}d_r(p_X,q_X)$ and $A_4$ be the point of the vertical geodesic $(V_{q_X},V_{q_Y})$ at the same height $h(q)+\frac{1}{2}d_r(p_X,q_X)$. Then $\gamma:=\gamma_1\cup\gamma_2\cup\gamma_3\cup\gamma_4\cup\gamma_5$ is constructed as follows:
\\\\\hspace*{0.6cm}- $\gamma_1$ is the part of $(V_{p_X},V_{p_Y})$ linking $p$ to $A_1$.
\\\hspace*{0.6cm}- $\gamma_2$ is a geodesic linking $A_1$ to $A_2$. Such a geodesic exists by Property \ref{HoroProdConnected}.
\\\hspace*{0.6cm}- $\gamma_3$ is the part of $(V_{p_X},V_{q_Y})$ linking $A_2$ to $A_3$.
\\\hspace*{0.6cm}- $\gamma_4$ is a geodesic linking $A_3$ to $A_4$. Such a geodesic exists by Property \ref{HoroProdConnected}.
\\\hspace*{0.6cm}- $\gamma_5$ is the part of $(V_{q_X},V_{q_Y})$ linking $A_4$ to $q$.
\\\\In fact $A_1$ and $A_2$ are close to each other. Indeed, the two points $A_1=(A_{1,X},A_{1,Y})$ and $A_2=(A_{2,X},A_{2,Y})$ are characterised by the two geodesics $(V_{p_X},V_{p_Y})$ and $(V_{p_X},V_{q_Y})$. Then, because $-h(q)=Y(q_Y)\leq Y(p_Y)$, Lemma \ref{LEM0} applied on $p_Y$ and $q_Y$ in $Y$ gives us $d_{Y}(A_{1,Y},A_{2,Y})\leq 288\delta$. Furthermore Property \ref{HoroProdConnected} provides us with $d_{\bowtie}\leq 2C_N(d_{X}+d_{Y})$, however we have that $A_{1,X}=A_{2,X}$ hence:
\begin{align}
d_{\bowtie}(A_1,A_2)\leq 576\delta C_N.\label{IneqA1A2}
\end{align}
Lemma \ref{LEM0} applied on $p_X$ and $q_X$ provides similarly:
\begin{align}
d_{\bowtie}(A_3,A_4)\leq 576\delta C_N,\label{IneqA3A4}
\end{align}
which gives us:
\begin{align*}
l_N(\gamma)=&l_N(\gamma_1)+l_N(\gamma_2)+l_N(\gamma_3)+l_N(\gamma_4)+l_N(\gamma_5)
\\=& d_{\bowtie}(p,A_1)+d_{\bowtie}(A_1,A_2)+d_{\bowtie}(A_2,A_3)+d_{\bowtie}(A_3,A_4)+d_{\bowtie}(A_4,q)
\\&\text{Since }\gamma_1,\ \gamma_3\text{ and }\gamma_5\text{ are vertical geodesics, we have:}
\\=& \Delta h(p,A_1)+d_{\bowtie}(A_1,A_2)+\Delta h(A_2,A_3)+d_{\bowtie}(A_3,A_4)+\Delta h(A_4,q)
\\=& \frac{1}{2}d_r(p_Y,q_Y)+d_{\bowtie}(A_1,A_2)+\frac{1}{2}d_r(p_Y,q_Y)+\frac{1}{2}d_r(p_X,q_X)+\Delta h(p,q)
\\&+d_{\bowtie}(A_3,A_4)+\frac{1}{2}d_r(p_X,q_X)
\\\leq& d_r(p_Y,q_Y)+d_r(p_X,q_X)+\Delta h(p,q)+1152\delta C_N\text{, by inequalities (\ref{IneqA1A2}) and (\ref{IneqA3A4}).}
\end{align*}

\begin{figure}
\begin{center}
\includegraphics[scale=1.6]{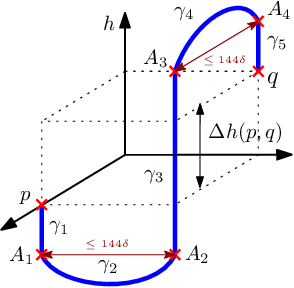}
\captionof{figure}{Construction of the path $\gamma$ when $h(p)\leq h(q)$ for Lemma \ref{PathConnectingPointsInHoroProduct}.}
\label{FigureConstructionPathHoroProduct}
\end{center}
\end{figure}
\end{proof}

We are aiming to use Proposition \ref{ExpLengthWhenBelowAndReachPoint} on the two components $\alpha_X\subset X$ and $\alpha_Y\subset Y$ of $\alpha$ to obtain lower bounds on their lengths. We hence need the following lemma to ensure us that when $\alpha$ is a geodesic, the exponential term in the inequality of Proposition \ref{ExpLengthWhenBelowAndReachPoint} will be small. 

\begin{lemma}\label{AnalyticStudy}
Let $C= 2853\delta C_N+2^{851}$ and let $e:\mathbb{R}\rightarrow\mathbb{R}$ be a map defined by $\forall t\in\mathbb{R}$, $e(t)=\frac{1}{C}2^{C^{-1}t}-2\max(0,t)$. Then $\forall t\in\mathbb{R}$:
\begin{enumerate}
\item $e(t)\geq -7C^2$
\item $(\ e(t)\leq 2853\delta C_N\ )\Rightarrow(\ t\leq 3C^2 \ )$.
\end{enumerate}
\end{lemma}
\begin{proof}
For all time $t$, we have that $e(t)=\frac{1}{C}2^{C^{-1}t}-2\max(0,t)\leq \frac{1}{C}2^{C^{-1}t}-2t=:e_1(t)$. The derivative of $e_1$ is $e_1'(t)=\frac{\log(2)}{C^2}2^{C^{-1}t}-2$, which is non negative $\forall t\geq C \log_2\left(\frac{2}{\log(2)}C^2\right)$ and non positive otherwise. Then $\forall t\in\mathbb{R}$:
\begin{align*}
e_1(t)&\geq e_1\left(\log_2\left(\frac{2}{\log(2)}C^2\right)\right)\geq \frac{2C}{\log(2)}-2C \log_2\left(\frac{2}{\log(2)}C^2\right)\geq \frac{2C}{\log(2)}-4C \log_2\left(\sqrt{\frac{2}{\log(2)}}C\right)
\\&\geq \frac{2C}{\log(2)}-4\sqrt{\frac{2}{\log(2)}}C^2\geq -4\sqrt{\frac{2}{\log(2)}}C^2\geq -7C^2.
\end{align*}
Since $C\geq \frac{2}{\log(2)}$ we have $3C^2\geq C\log_2(C^3)\geq C \log_2\left(\frac{2}{\log(2)}C^2\right)$, then $e_1$ is non decreasing on $[C\log_2(C^3);+\infty[$. We show that $e_1(3C^2)\geq 2853\delta C_N$:
\begin{align*}
e_1(3C^2)\geq e_1(C\log_2(C^3))=&\frac{1}{C}2^{\frac{C\log_2(C^3)}{C}}-2C\log_2(C^3)=C(C-6\log_2(C)).
\end{align*}
Since $C\geq 2^{851}$ we have $C-6\log_2(C)\geq 1$ and since $C\geq 2853\delta C_N$ we have that $e_1(3C^2)\geq C\times 1\geq 2853\delta C_N$ which provides $\forall t\in [3C^2;+\infty[$ we have $e_1(t)\geq 2853\delta C_N$. Furthermore $\forall t\in \mathbb{R}^+$, $e_1(t)=e(t)$, hence $\forall t\in [3C^2;+\infty[$ we have $e(t)\geq 2853\delta C_N$ which implies point $2.$ of this lemma.
\end{proof}

The following lemma provides us with a lower bound matching Lemma \ref{PathConnectingPointsInHoroProduct}, and a first control on the heights a geodesic segment must reach.

\begin{lemma}\label{LowerBoundLengthGeod}
Let $p=(p_X,p_Y)$ and $q=(q_X,q_Y)$ be two points of $X\bowtie Y$ such that $h(p)\leq h(q)$. Let $\alpha=(\alpha_X,\alpha_Y)$ be a geodesic segment of $ X\bowtie Y $ linking $p$ to $q$. Let $C_0=(2853\delta C_N+2^{851})^2$, we have:
\begin{enumerate}
\item $l(\alpha)\geq \Delta h (p,q)+d_r(p_Y,q_Y)+d_r(p_X,q_X)-15C_0$
\item $h^+(\alpha)\geq h(q)+\frac{1}{2}d_r(p_X,q_X)-3C_0$
\item $h^-(\alpha)\leq h(p)-\frac{1}{2}d_r(p_Y,q_Y)+3C_0$.
\end{enumerate}
\end{lemma}

\begin{proof}
Let us denote $\Delta H^+=h(q)+\frac{1}{2}d_r(p_X,q_X)-h^+(\alpha)$ and $\Delta H^-=h^-(\alpha)-\left(h(p)-\frac{1}{2}d_r(p_Y,q_Y)\right)$. Let $m$ be a point of $\alpha$ at height $h^-(\alpha)=h(p)-\frac{1}{2}d_r(p_Y,q_Y)+\Delta H^-$, and $n$ be a point of $\alpha$ at height $h^+(\alpha)=h(q)+\frac{1}{2}d_r(p_X,q_X)-\Delta H^+$. Then Proposition \ref{ExpLengthWhenBelowAndReachPoint} used on $\alpha_X$ gives us:
\begin{align*}
l(\alpha_X)\geq& 2\Delta h(p_X,m_X)+d(p_X,q_X)+2^{-850}2^{\frac{1}{\delta}\Delta H^+}-1-2\max(0,\Delta H^+)-1700\delta
\\\geq& 2h(p_X)-2\left(h(p_X)-\frac{1}{2}d_r(p_Y,q_Y)+\Delta H^-\right)+d(p_X,q_X)+2^{-850}2^{\frac{1}{\delta}\Delta H^+}-1
\\&-2\max(0,\Delta H^+)-1700\delta
\\\geq& d_r(p_Y,q_Y)+d_r(p_X,q_X)+\Delta h(p,q)+2^{-850}2^{\frac{1}{\delta}\Delta H^+}-1-2\max(0,\Delta H^+)-2\Delta H^--1700\delta.
\end{align*}
Since $h(p_Y)\geq h(q_Y)$ and $h(n_Y)=h(q_Y)-\frac{1}{2}d_r(p_X,q_X)+\Delta H^+$, Proposition \ref{ExpLengthWhenBelowAndReachPoint} used on $\alpha_Y$ provides similarly:
\begin{align*}
l(\alpha_Y)\geq d_r(p_X,q_X)+d_r(p_Y,q_Y)+\Delta h(p,q)+2^{-850}2^{\frac{1}{\delta}\Delta H^-}-1-2\max(0,\Delta H^-)-2\Delta H^+-1700\delta.
\end{align*}
Hence by Property \ref{ProprSplitLenght}:
\begin{align}
l_N(\alpha)\geq\frac{1}{2}(l(\alpha_X)+l(\alpha_Y))
\geq& d_r(p_X,q_X)+d_r(p_Y,q_Y)+\Delta h(p,q)-1700\delta+2^{-851}2^{\frac{1}{\delta}\Delta H^-}\nonumber\\&+2^{-851}2^{\frac{1}{\delta}\Delta H^+}-2\max(0,\Delta H^-)-2\max(0,\Delta H^+)-1.\label{UseEq1InThis}
\end{align}
Furthermore, we know by Lemma \ref{PathConnectingPointsInHoroProduct} that $l_N(\alpha)\leq \Delta h (p,q)+d_r(p_X,q_X)+d_r(p_Y,q_Y) +1152\delta C_N$. Since $C_N\geq1$ we have:
\begin{align*}
2852\delta C_N\geq& 2^{-851}2^{\frac{1}{\delta}\Delta H^-}-2\max(0,\Delta H^-)+2^{-851}2^{\frac{1}{\delta}\Delta H^+}-2\max(0,\Delta H^+)-1.
\end{align*}
Let us denote $S:=\max\{\Delta H^-,\Delta H^+\}$. Therefore we have $2^{-851}2^{\frac{1}{\delta}S}-2\max(0,S)-1\leq 2852\delta C_N$. By assumption $\delta\geq1$ hence $2^{-851}2^{\frac{1}{\delta}S}-2\max(0,S)\leq 2853\delta C_N$. Furthermore, for $C=2853\delta C_N+2^{851}$, we have both $2^{-851}\geq\frac{1}{C}$ and $\frac{1}{\delta}\geq \frac{1}{C}$. Then we have $\frac{1}{C}2^{\frac{S}{C}}-2\max(0,S)\leq 2853\delta C_N$. Lemma \ref{AnalyticStudy} provides $S\leq 3C^2=3C_0$ which implies points $2.$ and $3.$ of our lemma. Lemma \ref{AnalyticStudy} also provides us with:
\begin{align*}
-14C_0\leq &2^{-851}2^{\frac{1}{\delta}\Delta H^-}-2\max(0,\Delta H^-)+2^{-851}2^{\frac{1}{\delta}\Delta H^+}-2\max(0,\Delta H^+).
\end{align*}
Last inequality is a lower bound of the term we want to remove in inequality (\ref{UseEq1InThis}). The first point of our lemma hence follows since $1700\delta+1\leq C_0$.
\end{proof}

We recall that by definition:
\begin{align*}
&\forall p_X,q_X\in X,\ d_r(p_X,q_X)=d_{X}(p_X,q_X)-\Delta h(p_X,q_X)
\\&\forall p_Y,q_Y\in Y,\ d_r(p_Y,q_Y)=d_{Y}(p_Y,q_Y)-\Delta h(p_Y,q_Y)
\end{align*}
Hence combining Lemma \ref{PathConnectingPointsInHoroProduct} and \ref{LowerBoundLengthGeod} we get the following corollary.

\begin{cor}\label{lengthGeod}
Let $N$ be an admissible norm and let $C_0= (2853\delta C_N+2^{851})^2$. The length of a geodesic segment $\alpha$ connecting $p$ to $q$ in $( X\bowtie Y ,d_{\bowtie})$ is controlled as follows:
\begin{align*}
\big|l_N(\alpha)-\big(d_{X}(p_X,q_X)+d_{Y}(p_Y,q_Y)-\Delta h (p,q)\big)\big|\leq 15C_0,
\end{align*}
which gives us a control on the $N$-path metric, for all points $p$ and $q$ in $ X\bowtie Y $ we have:
\begin{align*}
\big|d_{\bowtie}(p,q)-\big(d_{X}(p_X,q_X)+d_{Y}(p_Y,q_Y)-\Delta h (p,q)\big)\big|\leq 15C_0.
\end{align*}
\end{cor}

This result is central as it shows that the shape of geodesics does not depend on the $N$-path metric chosen for the distance on the horospherical product.

\begin{cor}\label{l1samel2}
Let $r\geq 1$. For all $p$ and $q$ in $X\bowtie Y$ we have:
\begin{equation}
\big|d_{ \bowtie  ,\ell_r}(p,q)-d_{ \bowtie  ,\ell_1}(p,q)\big|\leq 30(5706\delta+2^{851})^2.\nonumber
\end{equation}
\end{cor}

\begin{proof}
The $\ell_r$ norm inequalities provide us with:
\begin{equation}
\sqrt[r]{{d_{X}}^r+{d_{Y}}^r}\leq d_{X}+d_{Y}\leq  2^{\frac{r-1}{r}}\sqrt[r]{{d_{X}}^r+{d_{Y}}^r}.\nonumber
\end{equation} 
Hence we have $\frac{\sqrt[r]{2}}{2}\left(d_{X}+d_{Y}\right)\leq  \sqrt[r]{{d_{X}}^r+{d_{Y}}^r}\leq d_{X}+d_{Y}$. Then the $\ell_r$ norms are admissible norms with $C_{\ell_r}\leq 2$, which ends the proof.
\end{proof}

The next corollary tells us that changing this distance does not change the large scale geometry of $ X\bowtie Y $. 

\begin{cor}
Let $N_1$ and $N_2$ be two admissible norms. Then the metric spaces $\left( X\bowtie Y , d_{ \bowtie  ,N_1}\right)$ and $\left( X\bowtie Y , d_{ \bowtie  ,N_2}\right)$ are roughly isometric.
\end{cor}

The control on the distances of Lemma \ref{lengthGeod} will help us understand the shape of geodesic segments and geodesic lines in a horospherical product.

\section{Shapes of geodesics and visual boundary of $ X\bowtie Y $}\label{SectionThm}
\subsection{Shapes of geodesic segments}

In this section we focus on the shape of geodesics. We recall that in all the following $X$ and $Y$ are assumed to be two proper, geodesically complete, $\delta$-hyperbolic, Busemann spaces with $\delta\geq 1$, and $N$ is assumed to be an admissible norm. 

The next lemma gives a control on the maximal and minimal height of a geodesic segment in a horospherical product. It is similar to the traveling salesman problem, who needs to walk from $x$ to $y$ passing by $m$ and $n$. This result follows from the inequalities on maximal and minimal heights of Lemma \ref{LowerBoundLengthGeod} combined with Lemma \ref{PathConnectingPointsInHoroProduct}.

\begin{lemma}\label{HPlusHMinusGeod}
Let $p=(p_X,p_Y)$ and $q=(q_X,q_Y)$ be two points of $X\bowtie Y$ such that $h(p)\leq h(q)$. Let $N$ be an admissible norm and let $\alpha=(\alpha_X,\alpha_Y)$ be a geodesic of $( X\bowtie Y ,d_{\bowtie})$ linking $p$ to $q$. Let $C_0=(2853\delta C_N+2^{851})^2$, we have:
\begin{enumerate}
\item $\left|h^-(\alpha) -\left(h(p)-\frac{1}{2}d_r(p_Y,q_Y)\right)\right|\leq 4C_0$
\item $\left|h^+(\alpha) -\left(h(q)+\frac{1}{2}d_r(p_X,q_X)\right)\right|\leq 4C_0$.
\end{enumerate}
\end{lemma}
\begin{proof}
Let us consider a point $m$ of $\alpha$ such that $h(m)=h^-(\alpha)$ and a point $n$ of $\alpha$ such that $h(n)=h^+(\alpha)$. Then $m$ comes before $n$ or $n$ comes before $m$. In both cases, since $h(m)\leq h(p)\leq h(q)\leq h(n)$ and by Lemma \ref{LemDistBigHaut} we have:
\begin{align*}
l_N(\alpha)&\geq \Delta h(p,q)+2(h(p)-h^-(\alpha))+2(h^+(\alpha)-h(q))
\\&\geq\Delta h(p,q)+2(h(p)-h^-(\alpha))+d_r(p_X,q_X)-6C_0,\text{ by Lemma \ref{LowerBoundLengthGeod}}.
\end{align*}
Furthermore Lemma \ref{PathConnectingPointsInHoroProduct} provides $l_N(\alpha)\leq \Delta h (p,q)+d_r(p_X,q_X)+d_r(p_Y,q_Y)+C_0$ , hence:
\begin{align*}
\Delta h (p,q)+d_r(p_X,q_X)+d_r(p_Y,q_Y)+C_0\geq\Delta h(p,q)+2(h(p)-h^-(\alpha))+d_r(p_X,q_X)-6C_0,
\end{align*}
which implies $\left(h(p)-\frac{1}{2}d_r(p_Y,q_Y)\right) -h^-(\alpha)\leq 4C_0$. In combination with the third point of Lemma \ref{LowerBoundLengthGeod} it proves the first point of our Lemma \ref{HPlusHMinusGeod}. The second point is proved similarly.
\end{proof}

\begin{lemma}\label{TechShapeGeod}
Let $N$ be an admissible norm and let $C_0=(2853\delta C_N+2^{851})^2$. Let $p=(p_X,p_Y)$ and $q=(q_X,q_Y)$ be two points of $X\bowtie Y$. Let $\alpha=(\alpha_X,\alpha_Y)$ be a geodesic of $( X\bowtie Y ,d_{\bowtie})$ linking $p$ to $q$. Then there exist two points $a=(a_X,a_Y),\ b=(b_X,b_Y)$ of $\alpha$ such that $h(a)=h(p)$, $h(b)=h(q)$ with the following properties:
\begin{enumerate}
\item If $h(p)\leq h(q)-7C_0$ then:
\begin{enumerate}
\item $h^-(\alpha)=h^-([x,a])$ and $h^+(\alpha)=h^+([b,y])$
\item $\left|d_r(p_Y,a_Y)-d_r(p_Y,q_Y)\right| \leq 16C_0\ \text{and}\ d_r(p_X,a_X)\leq 22C_0$
\item $\left|d_r(q_X,b_X)-d_r(p_X,q_X)\right| \leq 16C_0\ \text{and}\ d_r(q_Y,b_Y)\leq 22C_0$
\item $|d_{\bowtie}(a,b)-\Delta h(a,b)|\leq 13C_0$.
\end{enumerate}
\item If $h(q)\leq h(p)-7C_0$ then $(a)$, $(b)$, $(c)$ and $(d)$ hold by switching the roles of $p$ and $q$ and switching the roles of $a$ and $b$.
\item If $|h(p)-h(q)|\leq 7C_0$ at least one of the two previous conclusions is satisfied.
\end{enumerate}
\end{lemma}

Lemma \ref{TechShapeGeod} is illustrated in Figure \ref{NotaSectThm}. Its notations will be used in all section \ref{SectionThm}.

\begin{figure}
\begin{center}
\includegraphics[scale=0.8]{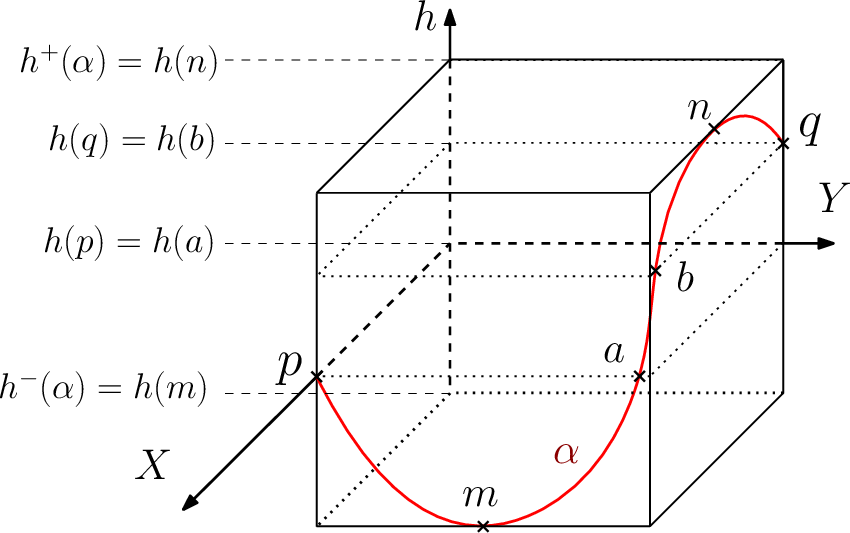}
\captionof{figure}{Notations of Lemma \ref{TechShapeGeod}.}
\label{NotaSectThm}
\end{center}
\end{figure}
\begin{proof}

Let us consider a point $m$ of $\alpha$ such that $h(m)=h^-(\alpha)$ and a point $n$ of $\alpha$ such that $h(n)=h^+(\alpha)$. We first assume that $m$ comes before $n$ in $\alpha$ oriented from $p$ to $q$. Let us call $a$ the first point between $m$ and $n$ at height $h(p)$ and $b$ the last point between $m$ and $n$ at height $h(q)$. Property $(a)$ of our Lemma is then satisfied. Let us denote $\alpha_1$ the part of $\alpha$ linking $p$ to $a$, $\alpha_2$ the part of $\alpha$ linking $a$ to $b$ and $\alpha_3$ the part of $\alpha$ linking $b$ to $q$. We have that $m$ is a point of $\alpha_1$ and that $n$ is a point of $\alpha_3$. Inequalities $2.$ and $3.$ of Lemma \ref{LowerBoundLengthGeod} used on $\alpha_1$ provide $l_N(\alpha_1)\geq d(p,m) +d(m,a)\geq 2\Delta h(p,m) \geq d_r(p_Y,q_Y)-6C_0$ and similarly $l_N(\alpha_3)\geq d_r(p_X,q_X)-6C_0$. Furthermore we have $l_N(\alpha_2)\geq \Delta h (p,q)$. Combining $l_N(\alpha_1)=l_N(\alpha)-l_N(\alpha_2)-l_N(\alpha_3)$ and Lemma \ref{PathConnectingPointsInHoroProduct} we have:
\begin{align}
l_N(\alpha_1)&\leq\Delta h (p,q)+d_r(p_X,q_X)+d_r(p_Y,q_Y)+C_0-\Delta h (p,q)-d_r(p_X,q_X)+6C_0\nonumber
\\&\leq d_r(p_Y,q_Y)+7C_0.\label{EquseInther2}
\end{align}
We have similarly that $l_N(\alpha_3)\leq d_r(p_X,q_X)+7C_0$ and that $d_{\bowtie}(a,b)=l_N(\alpha_2)\leq\Delta h (p,q) +13C_0$. It gives us $|d_{\bowtie}(a,b)-\Delta h (p,q)|\leq 13C_0$, point $(d)$ of our lemma. Furthermore, using Lemma \ref{HPlusHMinusGeod} on $\alpha$ and $\alpha_1$ provides:
\begin{align*}
\left|h^-(\alpha) -\left(h(p)-\frac{1}{2}d_r(p_Y,q_Y)\right)\right|&\leq 4C_0,
\\\left|h^-(\alpha_1) -\left(h(p)-\frac{1}{2}d_r(p_Y,a_Y)\right)\right|&\leq 4C_0.
\end{align*}
Since $h^-(\alpha)=h^-(\alpha_1)$ we have:
\begin{equation}
\left|d_r(p_Y,a_Y)-d_r(p_Y,q_Y)\right|\leq 16C_0,\label{EquseInther3}
\end{equation}
which is the first inequality of $(b)$. Using the first point of Lemma \ref{LowerBoundLengthGeod} on $\alpha_1$ in combination with inequality (\ref{EquseInther2}) gives us:
\begin{align*}
d_r(p_Y,q_Y)+7C_0\geq& l_N(\alpha_1)\geq \Delta h(p,a) +d_r(p_X,a_X)+d_r(p_Y,a_Y)-15C_0
\\\geq&d_r(p_X,a_X)+d_r(p_Y,a_Y)-15C_0
\\\geq&d_r(p_X,a_X)+d_r(p_Y,q_Y)-31C_0\text{, by inequality (\ref{EquseInther3}).}
\end{align*}
Then $d_r(p_X,q_X)\leq 38 C_0$ the second inequality of point $(b)$ holds. We prove similarly the inequality $(c)$ of this lemma. This ends the proof when $m$ comes before $n$. If $n$ comes before $m$, the proof is still working by orienting $\alpha$ from $q$ to $p$ hence switching the roles between $p$ and $q$. 
\\\\We will now prove that if $h(p)\leq h(q)-7C_0$ then $m$ comes before $n$ on $\alpha$ oriented from $p$ to $q$. Let us assume that $h(p)\leq h(q)-7C_0$. We will proceed by contradiction, let us assume that $n$ comes before $m$, using $h(m)\leq h(p)\leq h(q)\leq h(n)$ it implies:
\begin{align*}
l_N(\alpha)\geq& d_{\bowtie}(p,n)+d_{\bowtie}(n,m)+d_{\bowtie}(m,q) \geq  \Delta h(p,n)+\Delta h(n,m)+\Delta h(m,q)
\\\geq&\Delta h(p,q)+\Delta h(q,n)+\Delta h(m,p)+\Delta h(p,q)+\Delta h(q,n)+\Delta h(m,p)+\Delta h(p,q)
\\\geq&2\Delta h(p,q)+\Delta h(p,q)+2\Delta h(m,p)+2\Delta (q,n)
\\\geq&14C_0+\Delta h(p,q)+2(h(p)-h^-(\alpha))+2(h^+(\alpha)-h(q)).
\end{align*}
However Lemma \ref{LowerBoundLengthGeod} applied on $\alpha$ provides $h^+(\alpha)\geq h(q)+\frac{1}{2}d_r(p_X,q_X)-3C_0$ and $h^-(\alpha)\leq h(p)-\frac{1}{2}d_r(p_Y,q_Y)+3C_0$. Then:
\begin{align*}
l_N(\alpha)\geq& 14C_0+\Delta h(p,q) +d_r(p_X,q_X)+d_r(p_Y,q_Y)-12C_0
\\\geq&\Delta h(p,q) +d_r(p_X,q_X)+d_r(p_Y,q_Y)+2C_0,
\end{align*}
which contradict Lemma \ref{PathConnectingPointsInHoroProduct}. Hence, if $h(p)\leq h(q)-7C_0$, the point $m$ comes before the point $n$ and by the first part of the proof, $1.$ holds. Similarly, if $h(q)\leq h(p)-7C_0$ then $n$ comes before $m$ and then $2.$ holds. Otherwise when $|h(p)-h(q)|\leq 7C_0$ both cases could happened, then $1.$ or $2.$ hold.
\end{proof}

This previous lemma essentially means that if $p$ is sufficiently below $q$, the geodesic $\alpha$ first travels in a copy of $Y$ in order to "lose" the relative distance between $p_Y$ and $q_Y$, then it travels upward using a vertical geodesic from $a$ to $b$ until it can "lose" the relative distance between $p_X$ and $q_X$ by travelling in a copy of $X$. It looks like three successive geodesics of hyperbolic spaces, glued together. The idea is that the geodesic follows a shape similar to the path $\gamma$ we constructed in Lemma \ref{PathConnectingPointsInHoroProduct}. The following theorem tells us that a geodesic segment is in the constant neighbourhood of three vertical geodesics. It is similar to the hyperbolic case, where a geodesic segment is in a constant neighbourhood of two vertical geodesics. 

\begin{thm}\label{THMA}
Let $N$ be an admissible norm. Let $p=(p_X,p_Y)$ and $q=(q_X,q_Y)$ be two points of $X\bowtie Y$ and let $\alpha$ be a geodesic segment of $( X\bowtie Y ,d_{\bowtie})$ linking $p$ to $q$. Let $C_0=(2853\delta C_N+2^{851})^2$, there exist two vertical geodesics $V_1=(V_{1,X},V_{1,Y})$ and $V_2=(V_{2,X},V_{2,Y})$ such that:
\begin{enumerate}
\item If ~ $h(p)\leq h(q)-7C_0$ ~ then $\alpha$ is in the $196C_0C_N$-neighbourhood of $V_1\cup(V_{1,X},V_{2,Y})\cup V_2$
\item If ~ $h(p)\geq h(q)+7C_0$ ~ then $\alpha$ is in the $196C_0C_N$-neighbourhood of $ V_1\cup(V_{2,X},V_{1,Y})\cup V_2$
\item If ~ $|h(p)-h(q)|\leq 7C_0$ ~ then at least one of the conclusions of $1.$ or $2.$ holds.
\end{enumerate}
Specifically $V_1$ and $V_2$ can be chosen such that $p$ is close to $V_1$ and $q$ is close to $V_2$.
\end{thm}

\begin{figure}
\begin{center}
\includegraphics[scale=0.7]{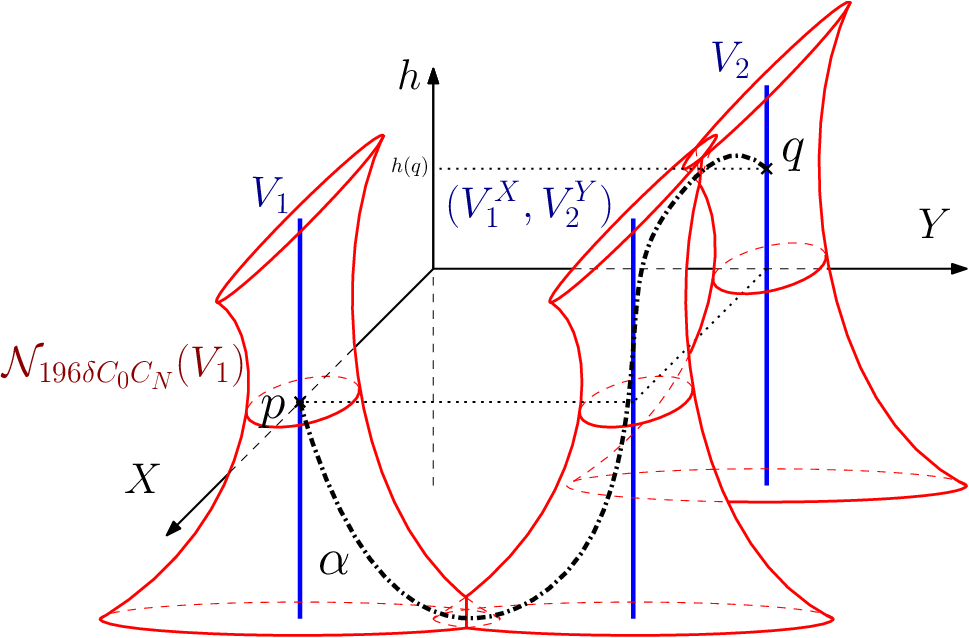}
\captionof{figure}{Theorem \ref{THMA}. The neighbourhood's shapes are distorted since when going upward, distances are contracted in the "direction" $X$ and expanded in the "direction" $Y$.}
\label{FigThmA}
\end{center}
\end{figure}

Figure \ref{FigThmA} pictures the $196C_0C_N$-neighbourhood of such vertical geodesics when $h(p)\leq h(q)-7C_0$.
When $|h(p)-h(q)|\leq 7C_0$, there are two possible shapes for a geodesic segment. In some cases, two points can be linked by two different geodesics, one of type $1$ and one of type $2$. 

\begin{proof}
Let $m=(m_X,m_Y)$ be a point of $\alpha$ such that $h(m)=h^-(\alpha)$, and $n=(n_X,n_Y)$ be a point of $\alpha$ such that $h(n)=h^+(\alpha)$. Then by Lemma \ref{HPlusHMinusGeod} we have:
\begin{equation}
\left|\Delta h(p,m)-\frac{1}{2}d_r(p_Y,q_Y)\right|\leq 4C_0.\label{UseAfter2}
\end{equation} 
We show similarly that:
\begin{equation}
\left|\Delta h(q,n)-\frac{1}{2}d_r(p_X,q_X)\right|\leq 4C_0.\label{UseAfter3}
\end{equation} 
In the first case we assume that $h(p)\leq h(q)-7C_0$. With notations as in Lemma \ref{TechShapeGeod}, and by inequality (\ref{EquseInther2}), we have that $l_N([p,a])\leq d_r(p_Y,q_Y)+7C_0$, hence:
\begin{align}
l_N([p,m])=&l_N([p,a])-l_N([a,m])\leq d_r(p_Y,q_Y)+7C_0-\Delta h(a,m)\nonumber
\\\leq&\frac{1}{2}d_r(p_Y,q_Y)+11C_0\text{, since }\Delta h(p,m)=\Delta h(a,m).\label{UseAfter1}
\end{align}
It follows from this inequality that:
\begin{align*}
d_{X}(p_X,m_X)=& 2d_{X\times Y}(p,m)-d_{Y}(p_Y,m_Y)\leq 2d_{\bowtie}(p,m)-d_{Y}(p_Y,m_Y)
\\\leq& 2l_N([p,m])-d_{Y}(p_Y,m_Y)\leq d_r(p_Y,q_Y)+22C_0-\Delta h(p,m)\leq \frac{1}{2}d_r(p_Y,q_Y)+26C_0.
\end{align*}
Then:
\begin{align*}
d_r(p_X,m_X)=& d_{X}(p_X,m_X)-\Delta h(p,m)\leq  \frac{1}{2}d_r(p_Y,q_Y)+26C_0-\Delta h(p,m)
\\\leq&30C_0 \text{, by inequality }(\ref{UseAfter2}).
\end{align*} 
Similarly $d_r(p_Y,m_Y)\leq 30C_0$. Let us consider the vertical geodesic $V_{m_X}$ of $X$ containing $m_X$, and the vertical geodesic $V_{p_Y}$ of $Y$ containing $p_Y$. Let us denote $p'_X$ the point of $V_{m_X}$ at the height $h(p)$. Since $d_r(p_X,m_X)\leq 30C_0$, Lemma \ref{LinkDrAndDSameHeight} applied on $p_X$ and $m_X$ provides $d_{X}(p_X,p'_X)\leq 31C_0$.
We will then consider two paths of $X$. The first one is $\alpha_{1,X}=[p_X,m_X]$, the part of $\alpha_X$ linking $p_X$ to $m_X$. The second one is $[m_X,p'_X]$ a piece of vertical geodesic linking $m_X$ to $p'_X$. We show that these two paths have close length. Using Property \ref{ProprSplitLenght} with inequalities (\ref{UseAfter2}) and (\ref{UseAfter1}) provides us with:
\begin{align*}
l_{X}([p_X,m_X])&\leq 2l_N([p,m])-l_{Y}([p_Y,m_Y])\leq 2\left(\frac{1}{2}d_r(p_Y,q_Y)+11C_0\right) -\Delta h(p,m)
\\&\leq\Delta h(p,m)+30C_0
\end{align*} 
Furthermore $l_{X}([p_X,m_X])\geq \Delta h(p,m)$ and we know that $l_{X}([m_X,p_X'])=\Delta h(p,m)$, hence:
\begin{align*}
\big\vert l_{X}([p_X,m_X])-l_{X}([m_X,p_X'])\big\vert\leq 30 C_0
\end{align*}
We already proved that their end points are also close to each other $d(p_X,p'_X)\leq 31C_0$. Since $\delta\leq C_0$, the property of hyperbolicity of $X$ gives us that $\alpha_{1,X}$ is in the $(31+30+1)C_0=62C_0$-neighbourhood of $[m_X,p'_X]$, a part of the vertical geodesic $V_{m_X}$. We show similarly that $\alpha_{1,Y}$ is in the $62C_0$-neighbourhood of $V_{p_Y}$. Since $N$ is an admissible norm, Property \ref{HoroProdConnected} gives us that $\alpha_{1}$ is in the $124C_0C_N$-neighbourhood of $(V_{m_X},V_{p_Y})$. We show similarly that $\alpha_3$, the portion of $\alpha$ linking $n$ to $q$, is in the $124C_0C_N$-neighbourhood of $(V_{q_X},V_{n_Y})$. 
We now focus on $\alpha_2$, the portion of $\alpha$ linking $m$ to $n$. Let us denote $[m_X,n_X]$ the path $\alpha_{2,X}$ and $[m_Y,n_Y]$ the path $\alpha_{2,Y}$. Then Lemma \ref{HPlusHMinusGeod} provides us with:
\begin{equation}
\left|\Delta h(m,n)-\Big(\Delta h(p,q)+\frac{1}{2}d_r(p_Y,q_Y)+\frac{1}{2}d_r(p_X,q_X)\Big)\right|\leq 8C_0.\label{UseAfter4}
\end{equation}
However from Lemma \ref{PathConnectingPointsInHoroProduct} and since $1152\delta C_N\leq C_0$:
\begin{align*}
l_N(\alpha_2)=&l_N(\alpha)-l_N(\alpha_1)-l_N(\alpha_3)
\\\leq& \Delta h(p,q)+d_r(p_X,q_X)+d_r(p_Y,q_Y)+C_0-\Delta h(p,m)-\Delta h(n,q)
\\\leq& \Delta h(p,q)+\frac{1}{2}d_r(p_X,q_X)+\frac{1}{2}d_r(p_Y,q_Y)+9C_0\text{, by inequalities }(\ref{UseAfter2})\text{ and }(\ref{UseAfter3}).
\end{align*}
It follows from this inequality and the fact that $N$ is admissible that:
\begin{align*}
d_{X}(m_X,n_X)&\leq 2l_N(\alpha_2)-d_{Y}(m_Y,n_Y)\leq 2\Delta h(p,q)+d_r(p_X,q_X)+d_r(p_Y,q_Y)+18C_0-\Delta h(m,n)
\\&\leq \Delta h(m,n)+34C_0\text{, by inequality }(\ref{UseAfter4}).
\end{align*}
Thus:
\begin{align*}
d_r(m_X,n_X)=& d_{X}(m_X,n_X)-\Delta h(m,n)\leq 34C_0 .
\end{align*} 
In the same way we have $d_r(m_Y,n_Y)\leq 34C_0$. Let us denote $n_X'$ the point of $V_{m_X}$ at the height $h(n_X)$. Since $d_r(p_X,m_X)\leq 34C_0$, Lemma \ref{LinkDrAndDSameHeight} applied on $m_X$ and $n_X$ provides:
\begin{equation}
d_{X}(m_X,n'_X)\leq 35C_0\label{UseAfter25}
\end{equation}
Hence we have proved that $\alpha_{2,X}$ and $[m_X,n_X']$ have their end points close to each other. Let us now prove that these paths have close lengths. We have that $l_{X}([m_X,n_X'])=\Delta h (m,n)$, and from inequalities (\ref{UseAfter2}) and (\ref{UseAfter3}) we have:
\begin{align*}
l_{X}([m_X,n_X])&\leq 2l_N(\alpha_{2,X})-l_{Y}([m_Y,n_Y])=2\Big(l_N(\alpha)-l_N(\alpha_1)-l_N(\alpha_3)\Big) -\Delta h(m,n)
\\&\leq 2\Big(15C_0+\Delta h (p,q) +d_r(p_X,q_X)+d_r(p_Y,q_Y)-\Delta h(p,m)-\Delta h(n,q)\Big) -\Delta h(m,n)
\\&\leq 2\Big(\Delta h (p,q) +d_r(p_X,q_X)+d_r(p_Y,q_Y)-\Delta h(p,m)-\Delta h(n,q)\Big) -\Delta h(m,n)
\\&\leq 2\Big(\Delta h (p,q) +\Delta h(p,m)+\Delta h(n,q)+16C_0\Big) -\Delta h(m,n)+30C_0\leq\Delta h(m,n)+62C_0
\end{align*}
As $l_{X}([m_X,n_X])\geq \Delta h(m,n)$ we obtain:
\begin{equation}
|l_{X}([m_X,n_X])-l_{X}([m_X,n_X'])|\leq 62 C_0\label{UseAfter26}
\end{equation} 
Then by similar arguments as for the path $\alpha_{1,X}$, inequalities (\ref{UseAfter25}) and (\ref{UseAfter26}) show that $\alpha_{2,X}$ is in the $(35+62+1)C_0=98C_0$ neighbourhood of $V_{m_X}$. Similarly we prove that $\alpha_{2,Y}$ is in the $98C_0$ neighbourhood of $V_{n_Y}$. Since $N$ is an admissible norm, Property \ref{HoroProdConnected} gives us that $\alpha_{2}$ is in the $196C_0C_N$-neighbourhood of $(V_{m_X},V_{n_Y})$.
\\\\In the second case, we assume that $h(q)\leq h(p)-7C_0$. Then by switching the role of $p$ and $q$, Lemma \ref{TechShapeGeod} gives us the result identically.
\\\\In the third case, we assume that $|h(p)- h(q)|\leq 7C_0$. Then Lemma \ref{TechShapeGeod} tells us that one of the two previous situations prevail, which proves the result.
\end{proof}

\subsection{Coarse monotonicity}

We will see that the following definition is related to being close to a vertical geodesic.

\begin{defn}
Let $C$ be a non negative number. A geodesic $\alpha:I\to X\bowtie Y $ of $X\bowtie Y$ is called $C$-coarsely increasing if $\forall t_1,t_2\in I$:
\begin{align*}
\big(\ t_2> t_1+C\ \big)\Rightarrow\big(\ h(\alpha(t_2))> h(\alpha(t_2))\ \big).
\end{align*}
The geodesic $\alpha$ is called $C$-coarsely decreasing if $\forall t_1,t_2\in I$:
\begin{align*}
\big(\ t_2> t_1+C\ \big)\Rightarrow\big(\ h(\alpha(t_2))< h(\alpha(t_2))\ \big).
\end{align*}
\end{defn}

The next lemma links the coarse monotonicity and the fact that a geodesic segment is close to vertical geodesics.

\begin{lemma}\label{ShapeCoarsMonotGeodHoroPoduct}
Let $N$ be an admissible norm and let $C_0=(2853\delta C_N+2^{851})^2$. Let $p=(p_X,p_Y)$ and $q=(q_X,q_Y)$ be two points of $X\bowtie Y$ and let $\alpha$ be a geodesic segment of $( X\bowtie Y ,d_{\bowtie})$ linking $p$ to $q$. Let $m\in\alpha$ and $n\in\alpha$ be two points in $ X\bowtie Y $ such that $h^-(\alpha)=h(m)$ and $h^+(\alpha)=h(n)$. We have:
\begin{enumerate}
\item  If ~ $h(p)\leq h(q)-7C_0$, then $\alpha$ is $17C_0$-coarsely decreasing on $[p,m]$ and $17C_0$-coarsely increasing on $[m,n]$ and $17C_0$-coarsely decreasing on $[n,q]$.
\item  If ~ $h(p)\geq h(q)+7C_0$, then $\alpha$ is $17C_0$-coarsely increasing on $[p,n]$ and $17C_0$-coarsely decreasing on $[n,m]$ and $17C_0$-coarsely increasing on $[m,q]$.
\item  If ~ $|h(p)-h(q)|\leq 7C_0$ ~ then the conclusions of $1.$ or $2.$ holds.
\end{enumerate}

\end{lemma}
\begin{proof}
Assume that $h(p)\leq h(q)-7C_0$. Then from inequality (\ref{UseAfter1}) in the proof of Theorem \ref{THMA}, $l_N([p,m])\leq \frac{1}{2} d_r(p_Y,q_Y)+11C_0$. Furthermore Lemma \ref{HPlusHMinusGeod} gives us that $\left|\Delta h(p,m)-\frac{1}{2}d_r(p_Y,q_Y)\right|\leq 4C_0$. Then:
\begin{align}
l_N([p,m])\leq \Delta h(p,m)+15C_0.\label{useInHere123}
\end{align}
We will proceed by contradiction, assume that $[p,m]$ is not $15C_0$-coarsely decreasing, then there exists $i_1\in\alpha$, $i_2\in\alpha$ such that $h(i_1)=h(i_2)$ and $l([i_1,i_2])> 15C_0$. Hence:
\begin{align*}
l_N([p,m])&\geq l_N([p,i_1])+l_N([i_1,i_2])+l_N([i_2,m])\geq\Delta h(p,i_1)+l_N([i_1,i_2])+\Delta h(i_2,m)
\\&>\Delta h(p,m)+15C_0,
\end{align*}
which contradicts inequality (\ref{useInHere123}). Then $[p,m]$ is $15C_0$-coarsely decreasing. We show in a similar way that $[m,n]$ is $17C_0$-coarsely increasing and that $[n,q]$ is $15C_0$-coarsely decreasing. This proves the first point of our lemma. The second point is proved by switching the roles of $p$ and $q$. We now assume $|h(p)-h(q)|\leq 7C_0$, as in the proof of Theorem \ref{THMA} the inequality (\ref{UseAfter1}) or a corresponding inequality holds, which ends the proof.
\end{proof}

\subsection{Shapes of geodesic rays and geodesic lines}

In this section we are focusing on using the previous results to get informations on the shapes of geodesic rays and geodesic lines. We first link the coarse monotonicity of a geodesic ray to the fact that it is close to a vertical geodesic. Let $\lambda\geq 1$ and $c\geq0$, a $(\lambda,c)$-quasigeodesic of the metric space $( X\bowtie Y ,d_{\bowtie})$ is the image of a function $\phi :\mathbb{R}\to  X\bowtie Y $ verifying that $\forall t_1,t_2\in\mathbb{R}$:
\begin{equation}
\frac{|t_1-t_2|}{\lambda}-c\leq d_{\bowtie}\big(\phi(t_1),\phi(t_2)\big)\leq\lambda|t_1-t_2|+c
\end{equation}

\begin{lemma}\label{ProjAreQuasiGeod}
Let $N$ be an admissible norm and let $C_0=(2853\delta C_N+2^{851})^2$. Let $\alpha=(\alpha_X,\alpha_Y)$ be a geodesic ray of $( X\bowtie Y ,d_{\bowtie})$ and let $K$ be a positive number such that $\alpha$ is $K$-coarsely monotone. Then $\alpha_X$ and $\alpha_Y$ are $(1,26C_0+8K)$-quasigeodesics. 
\end{lemma}
\begin{proof}
Let $t_1$ and $t_2$ be two times. Let us denote $p=(p_X,p_Y)=\alpha(t_1)$ and $q=(q_X,q_Y)=\alpha(t_2)$. We apply Lemma \ref{TechShapeGeod} on the part of $\alpha$ linking $p$ to $q$ denoted by $[p,q]$. By $K$-coarse monotonicity of $\alpha$ we have that $d(p,a)_{ X\bowtie Y ,N}\leq K$ and $d_{\bowtie}(b,q)\leq K$. Hence using $d)$ of Lemma \ref{TechShapeGeod}:
\begin{align*}
\Delta h(p,q)\leq d_{\bowtie}(p,q)&\leq d_{\bowtie}(p,a)+d_{\bowtie}(a,b)+d_{\bowtie}(b,q)\leq K+\Delta h(a,b)+13C_0+K
\\&\leq \Delta h(p,q)+\Delta h(p,a)+\Delta h(b,q)+13C_0+2K\leq\Delta h(p,q)+13C_0+4K.
\end{align*} 
Furthermore, $d_{X}(p_X,q_X)\geq \Delta h(p_X,q_X)=\Delta h(p,q)$ and $d_{Y}(p_Y,q_Y)\geq \Delta h(p,q)$. Since $N$ is an admissible norm we have:
\begin{align*}
\Delta h(p,q)\leq d_{X}(p_X,q_X)&=2d_{X\times Y}(p,q)-d_{Y}(p_Y,q_Y)\leq 2d_{\bowtie}(p,q)-d_{Y}(p_Y,q_Y)
\\&\leq 2\Delta h(p,q)+13C_0+4K-\Delta h(p,q)\leq\Delta h(p,q)+13C_0+4K .
\end{align*}
Hence:
\begin{align*}
d_{\bowtie}(p,q)-26C_0-8K\leq d_{X}(p_X,q_X)\leq d_{\bowtie}(p,q)+26C_0+8K ,
\end{align*}
By definition we have $p_X=\alpha_X(t_1)$, $q_X=\alpha_X(t_2)$ and $d_{\bowtie}(p,q)=|t_1-t_2|$. Then $\alpha_X$ is a $(1,26C_0+8K)$-quasigeodesic ray. We prove similarly that $\alpha_Y$ is a $(1,26C_0+8K)$-quasigeodesic ray.
\end{proof}

We will now make use of the rigidity property of quasi-geodesics in Gromov hyperbolic spaces, presented in Theorem 3.1 p.41 of \cite{Papa1}.

\begin{thm}[\cite{Papa1}]\label{THMPAPADO}
Let $H$ be a $\delta$-hyperbolic geodesic space. If $f:\mathbb{R}\rightarrow H$ is a $(\lambda,k)$-quasi geodesic, then there exists a constant $\kappa>0$ depending only on $\delta,\lambda$ and $k$ such that the image of $f$ is in the $\kappa$-neighbourhood of a geodesic in $H$.
\end{thm}

\begin{lemma}\label{LemCoarsClosVert}
Let $N$ be an admissible norm and let $T_1$ and $T_2$ be two real numbers. Let $\alpha=(\alpha_X,\alpha_Y):[T_1,+\infty[\to  X\bowtie Y $ be a geodesic ray of $( X\bowtie Y ,d_{\bowtie})$. Let $K$ be a positive number such that $\alpha$ is $K$-coarsely monotone. 
Then there exists a constant $\kappa>0$ depending only on $K$, $\delta$ and $N$ such that $\alpha$ is in the $\kappa$-neighbourhood of a vertical geodesic ray $V:[T_2;+\infty[\to  X\bowtie Y $ and such that $d_{\bowtie}\big(\alpha(T_1),V(T_2)\big)\leq\kappa$.
\end{lemma}

\begin{proof}
We assume without loss of generality that $\lim\limits_{t\rightarrow +\infty}h(\alpha(t))=+\infty$. Let $C_0=(2853\delta C_N+2^{851})^2$, by Lemma \ref{ProjAreQuasiGeod}, $\alpha_X$ is a $(1,26C_0+8K)$-quasi geodesic ray. Then Theorem \ref{THMPAPADO} says there exists $\kappa_X>0$ depending only on $26C_0+8K$ and $\delta$ such that $\alpha_X$ is in the $\kappa_X$-neighbourhood of a geodesic $V_X$. Since $C_0$ depends only on $\delta$ and $N$, $\kappa_X$ depends only on $K$, $\delta$ and $N$. Then $\lim\limits_{t\rightarrow +\infty}h(\alpha(t))=+\infty$ gives us $\lim\limits_{t\rightarrow +\infty}h(V_X(t))=+\infty$ which implies that $V_X$ is a vertical geodesic of $X$.
We will now build the vertical geodesic we want in $Y$. We have $\lim\limits_{t\rightarrow +\infty}h(\alpha_Y(t))=-\infty$ and by Lemma \ref{ProjAreQuasiGeod}:
\begin{align*}
\Delta h(\alpha_Y(t_1),\alpha_Y(t_2))-26C_0-8K\leq d_{Y}(\alpha_Y(t_1),\alpha_Y(t_2))\leq \Delta h(\alpha_Y(t_1),\alpha_Y(t_2))+26C_0+8K .
\end{align*}
Since $Y$ is Busemann, there exists a vertical geodesic ray $\beta$ starting at $\alpha_Y(T_1)$. Since $\beta$ is parametrised by its height, $\alpha_Y\cup\beta$ is also a $(1,26C_0+8K)$-quasi geodesic, hence there exists $\kappa_Y$ and $V_Y$ depending only on $K$, $\delta$ and $N$ such that $\alpha_Y\cup\beta$ is in the $\kappa_Y$-neighbourhood of $V_Y$. Since $\lim\limits_{t\rightarrow -\infty}h(V_Y(t))=+\infty$, $V_Y$ is a vertical geodesic of $Y$. 
\\Furthermore, by Property \ref{HoroProdConnected}, $d_{\bowtie}\leq 2C_N(d_{X}+d_{Y})$, hence there exists $\kappa$ depending only on $K$, $\delta$ and $N$ such that $\alpha$ is in the $\kappa$-neighbourhood (for $d_{\bowtie}$) of $(V_X,V_Y)$, a vertical geodesic of $( X\bowtie Y ,d_{\bowtie})$.  
Since $h(\alpha(t))\geq h(\alpha(T_1))-26C_0-8K=:M$, $\alpha$ is in the $\kappa$-neighbourhood of $\Big(V_X\big([M-\kappa;+\infty[\big),V_Y\big(]-\infty;-M+\kappa]\big)\Big)$ which is a vertical geodesic ray. 
\\\\We will now show that the starting points of $\alpha$ and $V$ are close to each other. Let us denote $T_1'$ a time such that $d_{\bowtie}\big(\alpha(T_1) , V(T_1')\big)\leq \kappa$, then $\Delta h \big(\alpha(T_1) , V(T_1')\big)\leq \kappa$, hence $|T_1'-M|\leq 26C_0+8K+\kappa$. Then by the triangle inequality: 
\begin{align*}
d_{\bowtie}\Big( \alpha(T_1), V(M-\kappa)\Big)\leq &d_{\bowtie}\Big( \alpha(T_1), V(T_1')\Big)+d_{\bowtie}\Big( V(T_1'), V(M-\kappa)\Big)
\\\leq& \kappa +26C_0+8K+\kappa + \kappa=26C_0+8K+3 \kappa
\end{align*} 
Let us denote $\kappa':=26C_0+8K+3 \kappa\geq\kappa$ and $T_2:=M-\kappa$. Hence $\alpha:[T_1;+\infty[\to X\bowtie Y $ is in the $\kappa'$-neighbourhood of a vertical geodesic ray $V:[T_2:+\infty[\to X\bowtie Y $, we have $d_{\bowtie}\big(\alpha(T_1),V(T_2)\big)\leq\kappa'$ and $\kappa'$ depends only on $\delta$ and $K$.
\end{proof}

\begin{lemma}\label{RayChangesOnce}
Let $N$ be an admissible norm and let $\alpha:\mathbb{R}^+\to X\bowtie Y $ be a geodesic ray of $( X\bowtie Y ,d_{\bowtie})$. Then $\alpha$ changes its $17C_0$-coarse monotonicity at most once.
\end{lemma}
\begin{proof}
Let $\alpha:\mathbb{R}^+\to X\bowtie Y $ be a geodesic ray. Thanks to Lemma \ref{ShapeCoarsMonotGeodHoroPoduct} $\alpha$ changes at most twice of $17C_0$-coarse monotonicity. Indeed, assume it changes three times, applying Lemma \ref{ShapeCoarsMonotGeodHoroPoduct} on the geodesic segment which includes these three times provides a contradiction. We will show in the following that it actually only changes once.
\\\hspace*{0.6cm}Assume $\alpha$ changes twice of $17C_0$-coarse monotonicity. Then $\alpha$ must be first $17C_0$-coarsely increasing or $17C_0$-coarsely decreasing. We assume without loss of generality that $\alpha$ is first $17C_0$-coarsely decreasing. Then there exist $t_1,t_2,t_3\in\mathbb{R}$ such that $\alpha$ is $17C_0$-coarsely decreasing on $[\alpha(t_1),\alpha(t_2)]$ then $17C_0$-coarsely increasing on $[\alpha(t_2),\alpha(t_3)]$ then $17C_0$-coarsely decreasing on $[\alpha(t_3),\alpha(+\infty)[$. 
\\Hence Lemma \ref{LemCoarsClosVert} applied on $[\alpha(t_3),\alpha(+\infty)[$ implies that there exists $\kappa>0$ depending only on $\delta$ (since the constant of coarse monotonicity depends only on $\delta$) and a vertical geodesic ray $V=(V_X,V_Y)$ such that $[\alpha(t_3),\alpha(+\infty)[$ is in the $\kappa$-neighbourhood of $V$. 
Since $h^+([\alpha(t_3),\alpha(+\infty)[)<+\infty$, we have that $\lim\limits_{t\rightarrow +\infty}h(\alpha(t))=-\infty$, hence there exists $t_4\geq t_3$ such that $h(\alpha(t_4))\leq h(\alpha(t_1))-7C_0$. Then  Lemma \ref{ShapeCoarsMonotGeodHoroPoduct} tells us that $\alpha$ is first $17C_0$-coarsely increasing, which contradicts what we assumed. 
\end{proof}

We have classified the possible shapes of geodesic rays. Since geodesic lines are constructed from two geodesic rays glued together, we will be able to classify their shapes too.

\begin{defn}\label{DefHpTypeHqType}
Let $N$ be an admissible norm and let $\alpha=(\alpha_X,\alpha_Y):\mathbb{R}\rightarrow  X\bowtie Y $ be a path of $( X\bowtie Y ,d_{\bowtie})$. Let $\kappa\geq 0$. 
\begin{enumerate}
\item $\alpha$ is called $X$-type at scale $\kappa$ if and only if:
\begin{enumerate}
\item $\alpha_X$ is in a $\kappa$-neighbourhood of a geodesic of $X$ 
\item $\alpha_Y$ is in a $\kappa$-neighbourhood of a vertical geodesic of $Y$.
\end{enumerate}
\item  $\alpha$ is called $Y$-type at scale $\kappa$ if and only if:
\begin{enumerate}
\item $\alpha_Y$ is in a $\kappa$-neighbourhood of a geodesic of $Y$ 
\item $\alpha_X$ is in a $\kappa$-neighbourhood of a vertical geodesic of $X$.
\end{enumerate}
\end{enumerate}
\end{defn}

The $X$-type paths follow geodesics of $X$, meaning that they are close to a geodesic in a copy of $X$ inside $ X\bowtie Y $. The $Y$-type paths follow geodesics of $Y$.

\begin{figure}
\begin{center}
\includegraphics[scale=1]{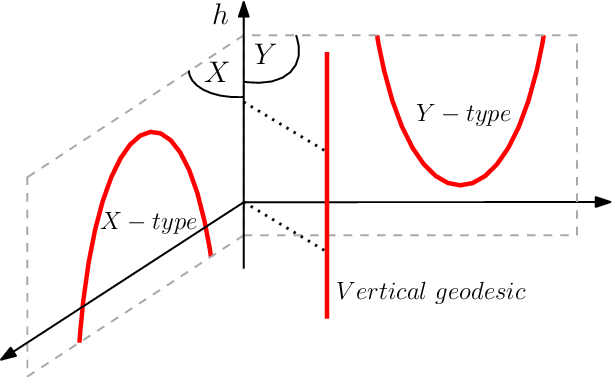}
\captionof{figure}{Different type of geodesics in $X\bowtie Y$.}
\label{FinalThm}
\end{center}
\end{figure}

\begin{rem}
In a horospherical product, being close to a vertical geodesic is equivalent to be both $X$-type and $Y$-type.
\end{rem}

\begin{thm}\label{THMB}
Let $N$ be an admissible norm. There exists $\kappa\geq 0$ depending only on $\delta$ and $N$ such that for any $\alpha:\mathbb{R}\rightarrow  X\bowtie Y $ geodesic of $( X\bowtie Y ,d_{\bowtie})$ at least one of the two following statements holds.
\begin{enumerate}
\item $\alpha$ is a $X$-type geodesic at scale $\kappa$ of $( X\bowtie Y ,d_{\bowtie})$
\item $\alpha$ is a $Y$-type geodesic at scale $\kappa$ of $( X\bowtie Y ,d_{\bowtie})$
\end{enumerate}
\end{thm}

\begin{proof}
It follows from Lemma \ref{RayChangesOnce} that $\alpha$ changes its coarse monotonicity at most once. Otherwise there would exist a geodesic ray included in $\alpha$ that changes at least two times of coarse monotonicity. We cut $\alpha$ in two coarsely monotone geodesic rays $\alpha_1:[0,+\infty[\rightarrow  X\bowtie Y $ and $\alpha_2:[0,+\infty[\rightarrow  X\bowtie Y $ such that up to a parametrisation $\alpha_1(0)=\alpha_2(0)$ and $\alpha_1\cup\alpha_2=\alpha$. 
By Lemma \ref{LemCoarsClosVert} there exists $\kappa_1$ and $\kappa_2$ depending only on $\delta$ such that $\alpha_1$ is in the $\kappa_1$-neighbourhood of a vertical geodesic ray $V_1=(V_{1,X},V_{1,Y}):[0;+\infty[\to  X\bowtie Y $ and such that  $\alpha_2$ is in the $\kappa_2$-neighbourhood of a vertical geodesic ray $V_2=(V_{2,X},V_{2,Y}):[0;+\infty[\to  X\bowtie Y $. This lemma also gives us $d_{\bowtie}\big(\alpha_1(0),V_1(0)\big)\leq\kappa_1$ and  $d_{\bowtie}\big(\alpha_2(0),V_2(0)\big)\leq\kappa_2$.
\\Assume that $\lim\limits_{t\rightarrow +\infty}h(V_{1,X}(t))=\lim\limits_{t\rightarrow +\infty}h(V_{2,X}(t))=+\infty$, then they are both vertical rays hence are close to a common vertical geodesic ray. Furthermore $\lim\limits_{t\rightarrow +\infty}h(V_{1,Y}(t))=\lim\limits_{t\rightarrow +\infty}h(V_{2,Y}(t))=-\infty$ in that case. Let $W_Y$ be the non continuous path of $Y$ defined as follows.
\begin{equation}
W_Y(t) = 
\left\{
    \begin{array}{ll}
        V_{1,Y}(-t) & \forall t\in ]-\infty;0]\\
        V_{2,Y}(t) &  \forall t\in ]0;+\infty[\\
    \end{array}
\right.\nonumber
\end{equation}
We now prove that $W_Y:\mathbb{R}\to Y$ is a quasigeodesic of $Y$. Let $t_1$ and $t_2$ be two real numbers. Since $V_{1,Y}$ and $V_{2,Y}$ are geodesics, $d_{Y}(W_Y(t_1),W_Y(t_2))=|t_1-t_2|$ if $t_1$ and $t_2$ are both non positive or both positive. Thereby we can assume without loss of generality that $t_1$ is non positive and that $t_2$ is positive. We also assume without loss of generality that $|t_1|\geq|t_2|$. The quasi-isometric upper bound is given by:
\begin{align*}
d_{Y}\big(W_Y(t_1),W_Y(t_2)\big)&=d_{Y}\big(V_{1,Y}(-t_1),V_{2,Y}(t_2)\big)\nonumber
\\&\leq d_{Y}\big(V_{1,Y}(-t_1),V_{1,Y}(0)\big)+d_{Y}\big(V_{1,Y}(0),V_{2,Y}(0)\big)+d_{Y}\big(V_{2,Y}(0),V_{2,Y}(t_2)\big)\nonumber
\\&\leq |t_1|+\kappa_1+\kappa_2+|t_2|
\\&\leq |t_1-t_2|+\kappa_1+\kappa_2,\text{ since }t_1\text{ and }t_2\text{ have different signs.}
\end{align*}
It remains to prove the lower bound of the quasi-geodesic definition on $W_Y$.
\begin{align}
d_{Y}\big(W_Y(t_1),W_Y(t_2)\big)&=d_{Y}\big(V_{1,Y}(-t_1),V_{2,Y}(t_2)\big)\nonumber
\\&\geq \frac{1}{2C_N}d_{\bowtie}\big(V_1(-t_1),V_2(t_2)\big)-d_{X}\big(V_{1,X}(-t_1),V_{2,X}(t_2)\big)\nonumber
\\&\geq \frac{1}{2C_N}d_{\bowtie}\big(\alpha(t_1),\alpha(t_2)\big)-\frac{\kappa_1+\kappa_2}{C_N}-d_{X}\big(V_{1,X}(-t_1),V_{2,X}(t_2)\big).\label{lala2}
\end{align}
The Busemann assumption on $X$ provides us with:
\begin{align*}
d_{X}\big(V_{1,X}(-t_1),V_{2,X}(-t_1)\big)\leq d_{X}\big(V_{1,X}(0),V_{2,X}(0)\big)\leq\kappa_1+\kappa_2.
\end{align*}
Since $\alpha$ is a geodesic and by using the triangle inequality on (\ref{lala2}) we have: 
\begin{align*}
d_{Y}\big(W_Y(t_1),W_Y(t_2)\big)&\geq \frac{|t_1-t_2|}{2C_N}-d_{X}\big(V_{1,X}(-t_1),V_{2,X}(-t_1)\big)-d_{X}\big(V_{2,X}(-t_1),V_{2,X}(t_2)\big)-\frac{\kappa_1+\kappa_2}{C_N}
\\&\geq \frac{|t_1-t_2|}{2C_N}-\Delta h\big(V_{2,Y}(-t_1),V_{2,Y}(t_2\big)-\left(\frac{1}{C_N}+1\right)(\kappa_1+\kappa_2).
\end{align*}
Assume that $\Delta h\big(V_{2,Y}(-t_1),V_{2,Y}(t_2)\big)\leq \frac{|t_1-t_2|}{4C_N}$, then: 
$$d_{Y}\big(W_Y(t_1),W_Y(t_2)\big)\geq \frac{|t_1-t_2|}{4C_N}-\left(\frac{1}{C_N}+1\right)(\kappa_1+\kappa_2).$$ 
Hence $W_Y$ is a $\left(\frac{1}{4C_N},\left(\frac{1}{C_N}+1\right)(\kappa_1+\kappa_2)\right)$ quasi-geodesic, which was the remaining case. Since $\kappa_1$ and $\kappa_2$ depend only on $\delta$ and $N$, there exists a constant $\kappa'$ depending only on $\delta$ and $N$ such that $V_{1,Y}\cup V_{2,Y}$ is in the $\kappa'$-neighbourhood of a geodesic of $Y$. The geodesic $\alpha$ is a $Y$-type geodesic in this case.
\\ Assume $\lim\limits_{t\rightarrow +\infty}h(V_{1,X}(t))=\lim\limits_{t\rightarrow +\infty}h(V_{2,X}(t))=-\infty$, we prove similarly that $\alpha$ is a $X$-type geodesic.
\end{proof}

If a geodesic is both $X$-type at scale $\kappa$ and $Y$-type at scale $\kappa$, then it is in a $\kappa$-neighbourhood of a vertical geodesic of $ X\bowtie Y $.

\subsection{Visual boundary of $ X\bowtie Y $}

We will now look at the visual boundary of our horospherical products. This notion is described for the Sol geometry in the work of Troyanov \cite{Troya} through the objects called geodesic horizons. We extend one of the definitions presented in page 4 of \cite{Troya} for horospherical products.

\begin{defn}
Two geodesics of a metric space $X$ are called asymptotically equivalent if they are at finite Hausdorff distance from each other. 
\end{defn}

\begin{defn}
Let $X$ be a metric space and let $o$ be a base point of $X$. The visual boundary of $X$ is the set of asymptotic equivalence classes of geodesic rays $\alpha:\mathbb{R^+}\rightarrow $ such that $\alpha(0)=o$, it is denoted by $\partial_o X$.
\end{defn}

We will use a result of \cite{Papa} to describe the visual boundary of horospherical products.

\begin{propr}[Property $10.1.7$ p.234 of \cite{Papa}]\label{PropertyPapadoRayStartPoint}
Let $X$ be a proper Busemann space, let $q$ be a point in $X$ and let $r:[0,+\infty[\to X$ be a geodesic ray. Then, there exists a unique geodesic ray $r'$ starting at $q$ that is asymptotic to $r$.  
\end{propr}

\begin{thm}\label{THMC}
Let $N$ be an admissible norm. We fix base points and directions $(w_{X},a_X)\in X\times\partial X$, $(w_{Y},a_Y)\in Y\times\partial Y$. Let $X\bowtie Y$ be the horospherical product with respect to $(w_{X},a_X)$ and $(w_{Y},a_Y)$. Then the visual boundary of $( X\bowtie Y ,d_{\bowtie})$ with respect to a base point $o=(o_X,o_Y)$ is given by:
\begin{align*}
\partial_o ( X\bowtie Y) =&\Big(\big(\partial X\setminus  \lbrace a_X\rbrace\big)\times\lbrace a_Y\rbrace\Big)\bigcup\Big(\lbrace a_X\rbrace\times\big(\partial Y\setminus \lbrace a_Y\rbrace\big) \Big)
\\=&\Big(\big(\partial X\times\lbrace a_Y\rbrace\big)\bigcup\big(\lbrace a_X\rbrace\times\partial Y\big)\Big)\setminus \lbrace(a_X,a_Y)\rbrace 
\end{align*}
\end{thm}

The fact that $(a_X,a_Y)$ is not allowed as a direction in $ X\bowtie Y $ is understandable since both heights in $X$ and $Y$ would tend to $+\infty$, which is impossible by the definition of $ X\bowtie Y $.

\begin{proof}
Let $\alpha$ be a geodesic ray. Lemma \ref{RayChangesOnce} implies that there exists $t_0\in\mathbb{R}$ such that $\alpha$ is coarsely monotone on $[t_0,+\infty[$. Then Lemma \ref{LemCoarsClosVert} tells us that $\alpha([t_0,+\infty[)$ is at finite Hausdorff distance from a vertical geodesic ray $V=(V_X,V_Y)$, hence $\alpha$ is also at finite Hausdorff distance from $V$. 
\\Since $X$ is Busemann and proper, Property \ref{PropertyPapadoRayStartPoint} ensure us there exists $V_X'$ a vertical geodesic ray such that $V_X$ and $V_X'$ are at finite Hausdorff distance with $V_X'(0)=o_{X}$. Similarly, there exists $V_Y'$ a vertical geodesic ray of $Y$ with $V_Y'(0)=o_Y$ such that $V_Y$ and $V_Y'$ are at finite Hausdorff distance. 
\\Furthermore, there is at least one vertical geodesic ray $V'=(V_Y',V_X')$ in every asymptotic equivalence class of geodesic rays, hence $\partial_o  X\bowtie Y $ is the set of asymptotic equivalence classes of vertical geodesic rays starting at $o$. Therefore, an asymptotic equivalence class can be identified by the couple of directions of a vertical geodesic ray. Then $\partial_o  X\bowtie Y $ can be identified to: 
\begin{equation}
\Big(\big(\partial X\setminus  \lbrace a_X\rbrace\big)\times\lbrace a_Y\rbrace\Big)\bigcup\Big(\lbrace a_X\rbrace\times\big(\partial Y\setminus \lbrace a_Y\rbrace\big) \Big).\nonumber
\end{equation}
the union between downward directions and upward directions, which proves the theorem.
\end{proof}

\begin{example}
In the case of Sol, $X$ and $Y$ are hyperbolic planes $\mathbb{H}_2$, hence their boundaries are $\partial X =\partial\mathbb{H}_2=S^1$ and $\partial Y =S^1$. Then $\partial_o \mathrm{Sol}$ can be identified to the following set: 
\begin{equation}
\big(S^1\setminus  \lbrace a_X\rbrace\big)\times\lbrace a_Y\rbrace\bigcup\lbrace a_X\rbrace\times\big(S^1\setminus \lbrace a_Y\rbrace\big) .
\end{equation}
It can be seen as two lines at infinity, one upward $\lbrace a_X\rbrace\times\big(S^1\setminus \lbrace a_Y\rbrace\big)$ and the other one downward $\big(S^1\setminus  \lbrace a_X\rbrace\big)\times\lbrace a_Y\rbrace$ .
\end{example}

It is similar to Proposition 6.4 of \cite{Troya}.

%\section{Possible further developments}
%
%A possible direction is the geometric study of multiple horospherical product such as $\{(x,y,z)\in X\times Y\times Z \mid h_X(x)+h_Y(y)+h_Z(z)=0\}$. Such horospherical products of trees have already been studied in \cite{BNW}. The techniques developed for the study of geodesic and virtual boundary, with some adaptations, might hold in this wider context.   
%\\\\Another development would be to remove the Busemann assumption. This assumption is not verified by finitely generated groups (other than free groups), and a description of some of their quasi-isometry group would be interesting in the view of their quasi-isometry classification.
%The Busemann assumption make the proofs less technical, which is appreciable since the proof are already quite technical. The generalisation would follow the same ideas, with an additional layer of coarse convexity. However even connectedness of a horospherical product is unclear without the Busemann assumption.

\end{document}